\documentclass[11pt]{article}
\usepackage{geometry}
\usepackage{amssymb,amsmath}
 \usepackage{pstricks, pst-coil, pst-node, pst-tree, multido}
 \usepackage[english]{babel}

\usepackage{graphics}
\usepackage{graphicx}
 \newtheorem{DE}{Definition}[section]

\newcommand {\sm} {\setminus}

\usepackage{latexsym} 
\usepackage{theorem} 
 \newcommand{\qed}{\relax\ifmmode\hskip2em\Box\else\unskip\nobreak\hfill$\Box$\fi}

\newtheorem{theorem}[DE]{Theorem}
\newtheorem{lemma}[DE]{Lemma}

{\theoremstyle{break}\theorembodyfont{\rmfamily}}
{\theoremstyle{break}\theorembodyfont{\rmfamily}}

\newcounter{claim}
\newenvironment{proof}[1][]%
	{\noindent {\setcounter{claim}{0}\sc proof --- }{#1}{}}{\qed\vspace{2ex}}
	{\refstepcounter{claim}\vspace{1ex}\noindent {(\it\arabic{claim}) {#1}{}}\it}{\vspace{1ex}}
	{\noindent {}{#1}{}}{ This proves~(\arabic{claim}).\vspace{1ex}}

\begin{document}

 \title{The (theta, wheel)-free graphs\\ Part II: structure theorem}
\author{Marko Radovanovi\'c\thanks{University
    of Belgrade, Faculty of Mathematics, Belgrade, Serbia. Partially
    supported by Serbian Ministry of Education, Science and
    Technological Development project 174033. E-mail:
    markor@matf.bg.ac.rs}~, Nicolas Trotignon\thanks{CNRS, LIP, ENS de
    Lyon. Partially supported by ANR project Stint under reference
    ANR-13-BS02-0007 and by the LABEX MILYON (ANR-10-LABX-0070) of
    Universit\'e de Lyon, within the program ‘‘Investissements
    d'Avenir’’ (ANR-11-IDEX-0007) operated by the French National
    Research Agency (ANR).  Also Universit\'e Lyon~1, universit\'e de
    Lyon. E-mail: nicolas.trotignon@ens-lyon.fr}~, Kristina Vu\v
  skovi\'c\thanks{School of Computing, University of Leeds, and
    Faculty of Computer Science (RAF), Union University, Belgrade,
    Serbia.  Partially supported by EPSRC grants EP/K016423/1 and EP/N0196660/1, and
    Serbian Ministry of Education and Science projects 174033 and
    III44006. E-mail: k.vuskovic@leeds.ac.uk}}

\maketitle

 \begin{abstract}
 A hole in a graph is a chordless cycle of length at least 4. A theta is a graph formed by three paths between the same
 pair of distinct vertices so that the union of any two of the paths induces a hole. A wheel is a graph formed by a hole
 and a node that has at least 3 neighbors in the hole. In this paper we obtain a decomposition theorem for the class
 of graphs that do not contain an induced subgraph isomorphic to a theta or a wheel, i.e.\ the class of (theta, wheel)-free
 graphs. The decomposition theorem uses clique cutsets and 2-joins.
 Clique cutsets are vertex cutsets that work really well in decomposition based algorithms, but are unfortunately not general enough to decompose more complex hereditary graph
 classes.
 A 2-join is an edge cutset that appeared in
 decomposition theorems of several complex classes, such as perfect graphs, even-hole-free graphs and others.
 In these decomposition theorems 2-joins are used together with vertex cutsets that are more general than clique cutsets, such as star cutsets and their generalizations
 (which are much harder to use in algorithms). This is a first example of a decomposition theorem that uses just the combination of
 clique cutsets and 2-joins. This has several consequences. First, we can easily transform our decomposition theorem into a complete structure
 theorem for (theta, wheel)-free graphs, i.e.\ we show how every (theta, wheel)-free graph can be built starting from basic graphs that can be explicitly
 constructed, and  gluing them together by prescribed composition operations; and all graphs built this way are (theta, wheel)-free.
 Such structure theorems are very rare for hereditary graph classes, only a few examples are known.
 Secondly, we obtain an $\mathcal O (n^4m)$-time decomposition based  recognition algorithm for (theta, wheel)-free graphs.
 Finally, in Parts III and IV of this series, we give further applications  of our decomposition theorem.
 \end{abstract}

 \section{Introduction}\label{sec:intro}
In this article, all graphs are finite and simple.

A \emph{prism} is a graph made of three node-disjoint chordless paths
$P_1 = a_1 \dots b_1$, $P_2 = a_2 \dots b_2$, $P_3 = a_3 \dots b_3$ of
length at least 1, such that $a_1a_2a_3$ and $b_1b_2b_3$ are triangles
and no edges exist between the paths except those of the two
triangles.  Such a prism is also referred to as a
$3PC(a_1a_2a_3,b_1b_2b_3)$ or a $3PC(\Delta ,\Delta )$ (3PC stands for
\emph{3-path-configuration}).

A \emph{pyramid} is a graph made of three chordless paths
$P_1 = a \dots b_1$, $P_2 = a \dots b_2$, $P_3 = a \dots b_3$ of
length at least~1, two of which have length at least 2, node-disjoint
except at $a$, and such that $b_1b_2b_3$ is a triangle and no edges
exist between the paths except those of the triangle and the three
edges incident to $a$.  Such a pyramid is also referred to as a
$3PC(b_1b_2b_3,a)$ or a $3PC(\Delta ,\cdot)$.

A \emph{theta} is a graph made of three internally node-disjoint
chordless paths $P_1 = a \dots b$, $P_2 = a \dots b$,
$P_3 = a \dots b$ of length at least~2 and such that no edges exist
between the paths except the three edges incident to $a$ and the three
edges incident to $b$.  Such a theta is also referred to as a
$3PC(a, b)$ or a $3PC(\cdot ,\cdot)$.

A \emph{hole} in a graph is a chordless cycle of length at least~4.  A
\emph{wheel} $W= (H, c)$ is a graph formed by a hole $H$ (called the
\emph{rim}) together with a node $c$ (called the \emph{center}) that
has at least three neighbors in the hole.

A \emph{3-path-configuration} is a graph isomorphic to a prism, a
pyramid or a theta.  Observe that the lengths of the paths in the
definitions of 3-path-configurations are designed so that the union of
any two of the paths induce a hole. A~\emph{Truemper configuration} is a
graph isomorphic to a prism, a pyramid, a theta or a wheel
(see Figure \ref{f:tc}).  Observe
that every Truemper configuration contains a hole.

\begin{figure}
  \begin{center}
    \includegraphics[height=2cm]{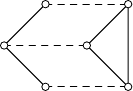}
    \hspace{.2em}
    \includegraphics[height=2cm]{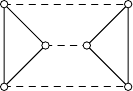}
    \hspace{.2em}
    \includegraphics[height=2cm]{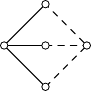}
    \hspace{.2em}
    \includegraphics[height=2cm]{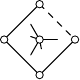}
  \end{center}
  \caption{Pyramid, prism, theta and wheel (dashed lines represent
    paths)\label{f:tc}}
\end{figure}

If $G$ and $H$ are graphs, we say that $G$ \emph{contains} $H$ when
$H$ is isomorphic to an induced subgraph of $G$.  We say that $G$ is
\emph{$H$-free} if it does not contain $H$.  We extend this to classes
of graphs with the obvious meaning (for instance, a graph is (theta,
wheel)-free if it does not contain a theta and does not contain a
wheel).

In this paper we prove a decomposition theorem for (theta, wheel)-free
graphs, from which we obtain a full structure theorem and a polynomial
time recognition algorithm.  This is part of a series of papers that
systematically study the structure of graphs where some Truemper
configurations are excluded. This project is motivated and explained
in more details in the first paper of the series~\cite{twf-p1}.
In Parts III and IV of the series (see \cite{twf-p3,twf-p4}) we give several applications of the
structure theorem.

\subsection*{The main result and the outline of the paper}

A graph is \emph{chordless} if all its cycles are chordless.   By the
following decomposition theorem proved in~\cite{twf-p1}, to prove a
decomposition theorem for (theta, wheel)-free graphs, it suffices to
focus on graphs that contain a pyramid.

\begin{theorem}[\cite{twf-p1}]\label{dt1-p1}
If $G$ is (theta, wheel, pyramid)-free, then $G$ is a line graph of a
triangle-free chordless graph or it has a clique cutset.
\end{theorem}

In Section \ref{sec:statement}, we define a generalization of pyramids
that we call P-graphs.  The full definition is complex, but
essentially, a P-graph is a graph that can be vertexwise partitioned
into the line graph of a triangle-free chordless graph and a
clique. Clearly, if a  (theta, wheel)-free graph contains a
pyramid, then it contains a P-graph.  We consider such a maximal
P-graph and prove that the rest of the graph attaches to it in a
special way that entails a decomposition.

The decompositions that we use are the clique cutset and the 2-join
(to be defined soon). Our main theorem is the following.

\begin{theorem}\label{main}
  If $G$ is (theta, wheel)-free, then $G$ is a line graph of a
  triangle-free chordless graph or a P-graph, or $G$ has a clique
  cutset or a 2-join.
\end{theorem}

Clique cutsets are vertex cutsets that work really well in decomposition based algorithms, but are unfortunately not general enough to decompose more complex hereditary graph
 classes.
 A 2-join is an edge cutset that appeared in
 decomposition theorems of several complex classes, such as perfect graphs \cite{crst}, even-hole-free graphs \cite{cckv:evenhole, dsv:ehf} and others.
 In these decomposition theorems 2-joins are used together with vertex cutsets that are more general than clique cutsets, such as star cutsets and their generalizations
 (which are much harder to use in algorithms). This is the first example of a decomposition theorem that uses just the combination of
 clique cutsets and 2-joins. This has several consequences. First, we can easily transform our decomposition theorem into a complete structure
 theorem for (theta, wheel)-free graphs, i.e.\ we show how every (theta, wheel)-free graph can be built starting from basic graphs that can be explicitly
 constructed, and  gluing them together by prescribed composition operations; and all graphs built this way are (theta, wheel)-free.
 Such structure theorems are very rare for hereditary graph classes, only a few examples are known, such as
 chordal graphs \cite{dirac:chordal}, universally-signable graphs \cite{cckv:universally}, graphs that do not contain a cycle with a unique chord \cite{nicolas.kristina:one},
 claw-free graphs \cite{chudnovsky.seymour:claw-free} and bull-free graphs \cite{chudnovsky:bull-free}
 (for a survey see \cite{kv-survey}).

The second consequence is the following theorem, and the remaining consequences are given in \cite{twf-p3}.

\begin{theorem}\label{th:reco}
  There exists an $O(n^4m)$-time algorithm that decides whether an
  input graph $G$ is (theta, wheel)-free.
\end{theorem}

In Section~\ref{sec:statement}, we give all the definitions needed in
the statement of Theorem~\ref{main}.  In particular, we define
P-graphs and 2-joins. In Section~\ref{sec:conn}, we study skeletons (the
skeleton is the root-graph of the line graph part of a P-graph). In
Section~\ref{sec:Pgraphs}, we study the properties of P-graphs. In
Section~\ref{sec:attach}, we study attachments to P-graphs in (theta,
wheel)-free graphs. In Section~\ref{sec:proofM}, we prove
Theorem~\ref{main}. In Section~\ref{sec:reco}, we prove
Theorem~\ref{th:reco} and describe how a structure theorem is derived from our decomposition theorem.

\subsection*{Terminology and notations}

A \emph{clique} in a graph is a (possibly empty) set of pairwise adjacent vertices. We
say that a clique is {\it big} if it is of size at least 3. A clique
of size~3 is also referred to as a {\em triangle}, and is denoted by
$\Delta$.  A {\it diamond} is a graph obtained from a clique of size 4
by deleting an edge. A {\it claw} is a graph induced by nodes
$u,v_1,v_2,v_3$ and edges $uv_1,uv_2,uv_3$.

A {\em path} $P$ is a sequence of distinct vertices
$p_1p_2\ldots p_k$, $k\geq 1$, such that $p_ip_{i+1}$ is an edge for
all $1\leq i <k$.  Edges $p_ip_{i+1}$, for $1\leq i <k$, are called
the {\em edges of $P$}.  Vertices $p_1$ and $p_k$ are the {\em ends}
of $P$.  A cycle $C$ is a sequence of vertices $p_1p_2\ldots p_kp_1$,
$k \geq 3$, such that $p_1\ldots p_k$ is a path and $p_1p_k$ is an
edge.  Edges $p_ip_{i+1}$, for $1\leq i <k$, and edge $p_1p_k$ are
called the {\em edges of $C$}.  Let $Q$ be a path or a cycle.  The
vertex set of $Q$ is denoted by $V(Q)$.  The {\em length} of $Q$ is
the number of its edges.  An edge $e=uv$ is a {\em chord} of $Q$ if
$u,v\in V(Q)$, but $uv$ is not an edge of $Q$. A path or a cycle $Q$
in a graph $G$ is {\em chordless} if no edge of $G$ is a chord of $Q$.

Let $A$ and $B$ be two disjoint node sets such that no node of $A$ is
adjacent to a node of $B$. A path $P=p_1 \ldots p_k$ {\em connects
  $A$ and $B$} if either $k=1$ and $p_1$ has neighbors in both $A$ and
$B$, or $k>1$ and one of the two endnodes of $P$ is adjacent to at
least one node in $A$ and the other endnode is adjacent to at least
one node in $B$. The path $P$ is a {\em direct connection between $A$
  and $B$} if in $G[V(P) \cup A \cup B]$ no path connecting $A$ and
$B$ is shorter than $P$.  The direct connection $P$ is said to be {\em
  from $A$ to $B$} if $p_1$ is adjacent to a node of $A$ and $p_k$ is
adjacent to a node of $B$.

Let $G$ be a graph.  For $x\in V(G)$, $N(x)$ is the set of all
neighbors of $x$ in $G$, and $N[x]=N(x) \cup \{ x\}$.  Let $H$ and $C$
be vertex-disjoint induced subgraphs of $G$.  The {\em attachment of
  $C$ over $H$}, denoted by $N_H(C)$, is the set of all vertices of
$H$ that have at least one neighbor in $C$.  When $C$ consists of a
single vertex $x$, we denote the attachment of $C$ over $H$ by
$N_H(x)$, and we say that it is an {\em attachment of $x$ over
  $H$}. Note that $N_H(x)=N(x)\cap V(H)$.
  For $S\subseteq V(G)$, $G[S]$ denotes the subgraph of $G$ induced by $S$.

When clear from the context, we will sometimes write $G$ instead of $V(G)$.

\section{Statement of the decomposition theorem}
\label{sec:statement}

We start by defining the cutsets used in the decomposition theorem.
In a graph $G$, a subset $S$ of nodes and edges is a {\em cutset} if
its removal yields a disconnected graph.
A node cutset $S$ is a {\em clique cutset} if $S$ is a
clique.
Note that every disconnected
graph has a clique cutset: the empty set.

For a graph $G$ and disjoint sets $A,B\subseteq V(G)$, we say that a node cutset $S$ of $G$ {\it separates $A$ and $B$} if $S\subseteq V(G)\setminus(A\cup B)$ and no vertex of $A$ is in the same connected component of $G\setminus S$ as some vertex of $B$.

An {\em almost 2-join} in a graph $G$ is a pair $(X_1,X_2)$ that is a
partition of $V(G)$, and such that:

\begin{itemize}
\item For $i=1,2$, $X_i$ contains disjoint nonempty sets $A_i$ and
  $B_i$, such that every node of $A_1$ is adjacent to every node
  of $A_2$, every node of $B_1$ is adjacent to every node of
  $B_2$, and there are no other adjacencies between $X_1$ and $X_2$.
\item For $i=1,2$, $|X_i|\geq 3$.
\end{itemize}

An almost 2-join $(X_1, X_2)$ is a \emph{2-join} when for $i\in\{1,2\}$, $X_i$
contains at least one path from $A_i$ to $B_i$, and if $|A_i|=|B_i|=1$
then $G[X_i]$ is not a chordless path.

We say that $(X_1,X_2,A_1,A_2,B_1,B_2)$ is a {\em split} of this 2-join, and the sets $A_1,A_2,B_1,B_2$ are the {\em special sets} of this 2-join.

A {\em star cutset} in a graph is a node cutset $S$ that contains a node (called a {\em center})
adjacent to all other nodes of $S$. Note that a nonempty clique cutset is a star cutset.

\begin{lemma}[\cite{twf-p1}]\label{star}
If $G$ is a  (theta, wheel)-free graph that has a star cutset, then $G$ has a clique cutset.
\end{lemma}

We now define the basic graphs.  A graph $G$ is {\em chordless} if no
cycle of $G$ has a chord, and it is {\em sparse} if for every edge
$e=uv$, at least one of $u$ or $v$ has degree at most~2.  Clearly all
sparse graphs are chordless.

An edge of a graph is {\em pendant} if at least one of its endnodes
has degree~1.  A \emph{branch vertex} in a graph is a vertex of degree
at least~3.  A {\em branch} in a graph $G$ is a path of length at
least~1 whose internal vertices are of degree 2 in $G$ and whose
endnodes are both branch vertices.  A {\em limb} in a graph $G$ is a
path of length at least~1 whose internal vertices are of degree 2 in
$G$ and whose one endnode has degree at least 3 and the other one has
degree~1. Two distinct branches are {\em parallel} if they have the
same endnodes.  Two distinct limbs are {\em parallel} if they share
the same vertex of degree at least 3.

Cut vertices of a graph $R$ that are also branch vertices are called
the {\em attaching vertices} of $R$.  Let $x$ be an attaching vertex
of a graph $R$, and let $C_1, \ldots ,C_t$ be the connected components
of $R\setminus x$ that together with $x$
are not limbs of $R$ (possibly, $t=0$, when all
connected components of $R\setminus x$ together with $x$ are limbs).  If $x$ is the end
of at least two parallel limbs of $R$, let $C_{t+1}$ be the subgraph of $R$ formed by
all the limbs of $R$ with endnode $x$.  The graphs
$R[V(C_i)\cup \{ x\} ]$ (for $i=1, \ldots, t$, if $t\neq 0$) and the graph $C_{t+1}$
(if it exists) are the \emph{$x$-petals} of $R$.

For any integer $k\geq 1$, a {\em $k$-skeleton}  is a graph $R$ such that (see Figures \ref{f:1-skeleton}, \ref{f:2-skeleton} and \ref{f:5-skeleton} for examples of $k$-skeletons for $k=1,2,5$):

\begin{enumerate}
\item\label{i:1} $R$ is connected, triangle-free, chordless and contains at least
  three pendant edges (in particular, $R$ is not a path).

 \item\label{i:2} $R$ has no parallel branches (but it may contains
   parallel limbs).

\item\label{i:3} For every cut vertex $u$ of $R$, every component of
  $R\setminus u$ has a vertex of degree 1 in $R$.

\item\label{i:4} For every vertex cutset $S=\{a, b\}$ of $R$ and for
  every component $C$ of $R\setminus S$, either $R[C\cup S]$ is a
  chordless path from $a$ to $b$, or $C$ contains at least one vertex
  of degree 1 in $R$.

\item\label{i:5} For every edge $e$ of a cycle of $R$, at least one of the
  endnodes of $e$ is of degree 2 in $R$.

\item\label{i:6} Each pendant edge of $R$ is given one label, that is an integer
  from $\{1, \dots, k\}$.

\item\label{i:7} Each label from $\{ 1, \ldots ,k\}$ is given at least once (as a
  label), and some label is used at least twice.

\item\label{i:8} If some pendant edge whose one endnode is of degree at least 3
  receives label $i$, then no other pendant edge receives label $i$.

\item\label{i:9}
If $R$ has no branches then $k=1$, and otherwise
  if two limbs of $R$ are parallel, then
  their pendant edges receive different labels and at least one of these labels is used more than once.

\item\label{i:10} If $k>1$ then for every attaching vertex $x$ and for
  every $x$-petal $H$ of $R$, there are at least two distinct labels
  that are used in $H$.  Moreover, if ${H}'$ is a union of at least one but not all  $x$-petals,
  then there is a label $i$ such that both ${H}'$ and $(R\setminus{H}')\cup\{x\}$ have
  pendant edges with label~$i$.

\item\label{i:11} If $k=2$, then both labels are used at least twice.
\end{enumerate}

Note that if  $R$ is a $k$-skeleton, then it edgewise
partitions into its branches and its limbs. To prove this, let $e$ be an edge of $R$ and $P=u\ldots v$, where $\deg(u)\geq\deg(v)$, the maximal path of $R$ that contains $e$ and whose internal vertices are of degree 2. If $\deg(u)=\deg(v)=1$, then $R$ is the chordless path induced by $V(P)$, which contradicts \ref{i:1}. If $\deg(v)=2$, then, by the maximality of $P$, $uv$ is an edge of $R$. Now, if $\deg(u)=2$, then $R$ is the chordless cycle induced by $V(P)$, which contradicts \ref{i:1}; if $\deg(u)\geq 3$, then  $u$ is a cut vertex of $R$ that contradicts \ref{i:3}. So, $\deg(u),\deg(v)\geq 3$ and $P$ is a branch of $R$, or $\deg(u)\geq 2$ and $\deg(v)=1$ in which case $\deg (u)\geq 3$ (by the maximality of $P$) and $P$ is a limb of $R$.

Also, there is a trivial
one-to-one correspondence between the pendant edges of $R$ and the
limbs of $R$: any pendant edge belongs to a unique limb, and
conversely any limb contains a unique pendant edge.

If $R$ is a graph, then the {\em line graph} of $R$, denoted by $L(R)$, is the graph whose nodes are
the edges of $R$ and such that two nodes of $L(R)$ are adjacent in $L(R)$ if and only if the corresponding
edges are adjacent in $R$.

A {\em P-graph}  is any graph $B$ that can be
constructed as follows (see Figures \ref{f:1-skeleton}, \ref{f:2-skeleton} and \ref{f:5-skeleton} for examples of P-graphs):

\begin{itemize}
\item
Pick an integer $k\geq 1$ and a $k$-skeleton $R$.
\item
Build $L(R)$, the line graph of $R$. The vertices of $L(R)$ that
correspond to pendant edges of $R$ are called {\em pendant vertices}
of $L(R)$, and they receive the same label as their corresponding
pendant edges in $R$.
\item Build a clique $K$ with vertex set $\{ v_1, \ldots ,v_k\}$,
  disjoint from $L(R)$.
\item $B$ is now constructed from $L(R)$ and $K$ by adding edges
  between $v_i$ and all pendant vertices of $L(R)$ that have label
  $i$, for $i=1, \ldots ,k$.
\end{itemize}

We say that $K$ is the {\em special clique}
of $B$ and $R$ is the {\em skeleton} of $B$.

The next lemma, that is proved in Part I, allows us to focus on (theta, wheel, diamond)-free
graphs in the remainder of the paper.

\begin{lemma}[\cite{twf-p1}]\label{diamond}
  If $G$ is a wheel-free graph that contains a diamond, then $G$ has a
  clique cutset.
\end{lemma}

Observe that P-graphs are generalizations of pyramids (this
is why we call them P-graphs).  Let us explain this. A pyramid is {\em
  long} if all of its paths are of length greater than~1.  Note that
in a wheel-free graph all pyramids are long.  Every long pyramid
$\Pi =3PC(x_1x_2x_3, y)$ is a P-graph, where $K=\{ y\}$ and $R$ is a
tree that is obtained from a claw by subdividing each edge at least
once and giving all pendant edges label 1 (see Figure \ref{f:1-skeleton}).  It can be checked that
a pyramid whose one path is of length~1 (and that is therefore a
wheel) is not a P-graph. This is a consequence of Lemma~\ref{l3} to be
proved soon, but let us sketch a direct proof: the apex of the pyramid
is the center of a claw, so it must be in the special clique, which
therefore has size 1 or 2.  It follows that the skeleton must contain
two pendant edges with the same label, and one of them contains a
vertex of degree~3, a contradiction to condition~\ref{i:8}.

\begin{lemma}\label{pyrBas}
  A long pyramid is a P-graph.
\end{lemma}

In fact, every P-graph contains a long pyramid.  Formally we do not
need this simple fact, we therefore just sketch the proof: consider
three pendant edges of the skeleton for which at most two labels are
used (this exists by~(i) and \ref{i:7}).  Consider a minimal
connected subgraph $T$ of $R$ that contains these three edges.  It is
easy to check that $T$ is a tree with three pendant edges and a unique
vertex $v$ of degree~3, and that adding to its line graph the vertices
of $K$ corresponding to  at most two labels yields a long pyramid.
To check that the pyramid is long condition \ref{i:8} is used, to
check that two paths of $T$ linking $v$ to pendant edges with the same
label have length at least 2.

\begin{figure}
  \begin{center}
    \includegraphics[height=2.5cm]{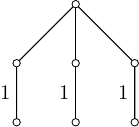}
    \hspace{2.5em}
    \includegraphics[height=2.5cm]{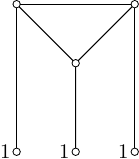}
    \hspace{2.5em}
    \includegraphics[height=2.5cm]{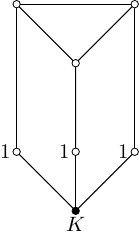}
    \caption{A 1-skeleton, its line graph and the corresponding P-graph.\label{f:1-skeleton}}
  \end{center}
    \begin{center}
    \includegraphics[height=2cm]{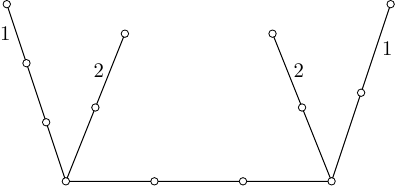}
    \hspace{1.5em}
    \includegraphics[height=2cm]{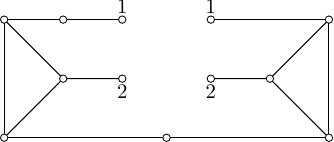}\\[3mm]
    \includegraphics[height=2cm]{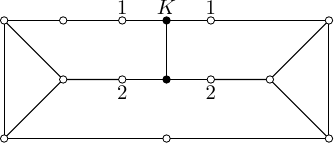}
  \end{center}
  \caption{A 2-skeleton, its line graph and the corresponding P-graph.\label{f:2-skeleton}}
\end{figure}
\begin{figure}
    \begin{center}
    \includegraphics[height=6.5cm]{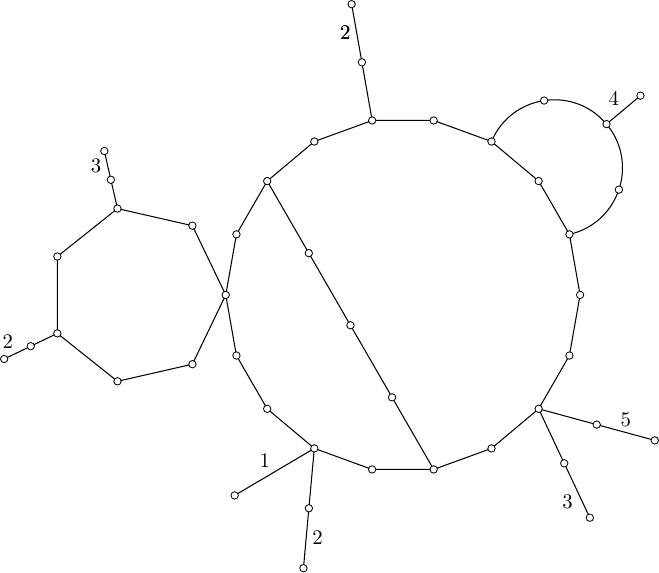}
    \hspace{1.5em}
    \includegraphics[height=6.5cm]{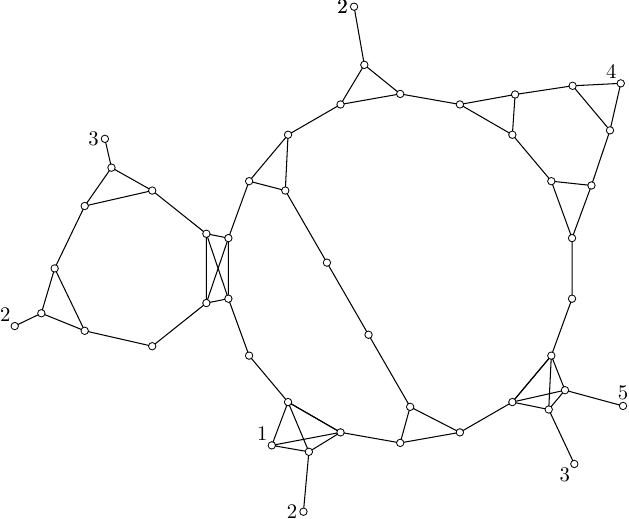}\\[3mm]
    \includegraphics[height=6.5cm]{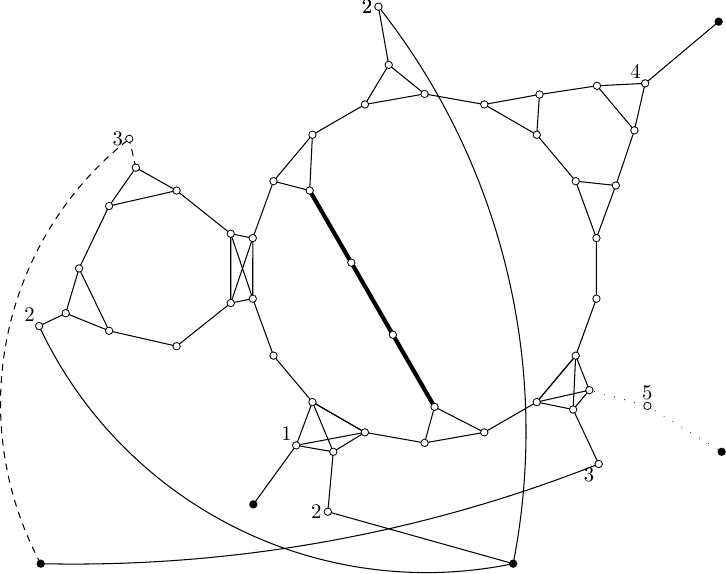}
  \end{center}
  \caption{A 5-skeleton, its line graph and the corresponding P-graph (black vertices are the vertices of the special clique of this P-graph; edges between them are not drawn). An internal (resp.\ claw, clique) segment of this P-graph is represented by a bold (resp.\ dashed, dotted) line.\label{f:5-skeleton}}
\end{figure}

\pagebreak

\section{Connectivity of skeletons}
\label{sec:conn}

In the following theorem we state versions of Menger's theorem that we use in this paper.

\begin{theorem}
Let $G$ be a graph.
\begin{itemize}
\item[(i)] Let $u$ and $v$ be non-adjacent vertices of $G$. Then the maximum number of internally vertex-disjoint paths from $u$ to $v$ is equal to the minimum size of a cutset $S$ of $G$ that separates $\{u\}$ and $\{v\}$.
\item[(ii)] Let $A$ and $B$ be disjoint subsets of $V(G)$. Then the maximum number of vertex-disjoint paths with one endnode in $A$ and the other in $B$ is equal to the minimum size of a cutset $S$ of $G$ that separates $A$ and $B$.
\item[(iii)] Let  $u\in V(G)$ and $B\subseteq V(G)\setminus\{u\}$. Then the maximum number of paths from $u$ to  $B$ that are vertex-disjoint except at $u$ is equal to the minimum size of a cutset $S$ of $G$ that separates $\{u\}$ and~$B$.
\end{itemize}
\end{theorem}

Additionally, we will often use the following variant of Menger's theorem, which is due to Perfect \cite{Perfect}.

Let $G$ be a graph, $x\in V(G)$ and $Y\subseteq V(G)\setminus\{x\}$. A set of $k$ paths $P_1,P_2,\ldots,P_k$ of $G$ is a {\em $k$-fan from $x$ to $Y$} if $V(P_i)\cap V(P_j)=\{x\}$, for $1\leq i<j\leq k$, and $|V(P_i)\cap Y|=1$, for $1\leq i\leq k$. A {\em fan from $x$ to $Y$} is a $|Y|$-fan from $x$ to $Y$.

\begin{lemma}[\cite{bondy.murty:book,Perfect}]\label{Perfect}
Let $G$ be a graph, $x\in V(G)$ and $Y,Z\subseteq V(G)\setminus\{x\}$ such that $|Y|<|Z|$. If there are fans from $x$ to $Y$ and from $x$ to $Z$, then there is a fan from $x$ to $Y\cup\{z\}$, for some $z\in Z\setminus Y$.
\end{lemma}

For distinct vertices  $v_1,v_2,\ldots,v_k$ of $G$, and pairwise disjoint and non-empty subsets $W_1,W_2,\ldots,W_k$ of $V(G)\setminus\{v_1,v_2,\ldots,v_k\}$,
we say that $k$ vertex-disjoint  paths $P_1,P_2,\ldots,P_k$ are {\em from $\{v_1,v_2,\ldots,v_k\}$ to $\{W_1,W_2,\ldots,W_k\}$} if for some permutation
$\sigma\in\mathbb S_k$,  $P_i\cap \{v_1,v_2,\ldots,v_k\}=\{v_i\}$ and $P_i\cap (W_1\cup W_2\cup\ldots\cup W_k)$ is a vertex of $W_{\sigma(i)}$, for $1\leq i\leq k$.

\begin{lemma}\label{MengerModified}
Let $G$ be a connected graph, $v_1,v_2,\ldots,v_k$ distinct vertices of $G$ and $W_1,W_2,\ldots,W_k$ pairwise disjoint and non-empty subsets of $V(G)\setminus\{v_1,v_2,\ldots,v_k\}$, such that all vertices of $W_1$ are of degree 1. The following holds:
\begin{itemize}
  \item[(1)] if $k=2$, and all vertices of $W_2$ are of degree 1 or $W_2=\{w_2\}$, then there exist 2 vertex-disjoint  paths from $\{v_1,v_2\}$ to $\{W_1,W_2\}$, or a vertex $u$ that separates $\{v_1,v_2\}$ from $W_1\cup W_2$;
  \item[(2)] if $k=3$, $W_2=\{w_2\}$, $W_3=\{w_3\}$ and there exist 2 vertex-disjoint  paths from $\{v_2,v_3\}$ to $\{w_2,w_3\}$, then there exist 3 vertex-disjoint  paths from $\{v_1,v_2,v_3\}$ to $\{W_1,\{w_2\},\{w_3\}\}$, or there exist vertices $u_1$ and $u_2$ such that $\{u_1,u_2\}$ separates $\{v_1,v_2,v_3\}$ from $W_1\cup\{w_2,w_3\}$.
\end{itemize}
\end{lemma}

\begin{proof}
Let $G'$ be the graph obtained from $G$ by adding a vertex $v$ ($v\not\in V(G)$) and edges $vv_i$, for $1\leq i\leq k$.

\medskip

\noindent (1) By Menger's theorem, there is a vertex $u$ that
separates $\{v_1,v_2\}$ from $W_1\cup W_2$, or two vertex-disjoint
paths from $\{v_1,v_2\}$ to $W_1\cup W_2$. If the first outcome holds,
then we are done, so we may assume that there are vertex-disjoint
paths $P_1$ and $P_2$ from $\{v_1,v_2\}$ to $W_1\cup W_2$. If both
$W_1$ and $W_2$ contain an endnode of $P_1$ and $P_2$, then we are
again done.  So we assume that both $P_1$ and $P_2$ have an endnode in
w.l.o.g.\ $W_1$, and let these endnodes be $v_1'$ and $v_2'$. This
means that in $G'$ there is a fan from $v$ to $\{v_1',v_2'\}$. Since
$G'$ is connected, there is a fan from $v$ to some $v''\in W_2$, and
therefore, by Lemma \ref{Perfect}, there is a fan from $v$ to
$\{v',v''\}$, for some $v'\in\{v_1',v_2'\}$. This completes the proof of (1).

\medskip

\noindent (2) By Menger's theorem, there are vertices $u_1$ and $u_2$
that separate $\{v_1,v_2,v_3\}$ from $W_1\cup\{w_2,w_3\}$, or three
vertex-disjoint  paths such that each of them has one endnode in
$\{v_1,v_2,v_3\}$ and the other in $W_1\cup\{w_2,w_3\}$. If the first
outcome holds, then we are done, so we may assume that there are vertex-disjoint  paths $P_1,P_2,P_3$ such that each of them has one endnode in
$\{v_1,v_2,v_3\}$ and the other in $W_1\cup\{w_2,w_3\}$. Let the
endnodes of paths $P_1,P_2,P_3$ that are in $W_1\cup\{w_2,w_3\}$ be
$v_1',v_2',v_3'$. This means that in $G'$ there is a fan from $v$ to
$\{v_1',v_2',v_3'\}$. By the conditions of the lemma, there is also a
fan from $v$ to $\{w_2,w_3\}$, and therefore, by Lemma \ref{Perfect},
there is a fan from $v$ to $\{w,w_2,w_3\}$, for some
$w\in\{v_1',v_2',v_3'\}\setminus\{w_2,w_3\}$. Since
$\{v_1',v_2',v_3'\}\setminus\{w_2,w_3\}$ is a subset of $W_1$, this
completes our proof.
\end{proof}

 Recall a standard notion: a
\emph{block} of a graph is an induced subgraph that is connected,  has no cut vertices
and is maximal with respect to these properties.  Recall that every block
of a graph is either 2-connected, or is a single edge.  Recall that
cut vertices of a graph $R$ that are of degree at least 3 are called the {\em
  attaching vertices} of $R$.

\begin{lemma}\label{l11}
  Let $R$ be a $k$-skeleton.  If $C$ is a 2-connected block of $R$, then
  no two vertices of $C$ that are of degree at least 3 in $R$ are
  adjacent.  In particular, every 2-connected block of $R$ is sparse,
  no two adjacent vertices of every cycle of $R$ have degree at least
  3, and if an edge of $R$ is between two vertices of degree at
  least~3, then it is a cutedge of $R$.
\end{lemma}
\begin{proof}
  This is equivalent to condition~\ref{i:5} in the definition of a
  P-graph, since an edge of $R$ belongs to a cycle if and only if it
  belongs to a 2-connected block of $R$.
\end{proof}

\begin{lemma}\label{l4}
  Let $R$ be a $k$-skeleton.  If $e_1$ and $e_2$ are edges of $R$, then
  there exists a cycle of $R$ that goes through $e_1$ and $e_2$, or
  there exists a path in $R$ whose endnodes are of degree~1 (in
  $R$) and that goes through $e_1$ and $e_2$.
\end{lemma}

\begin{proof}
  We set $e_1 = u_1v_1$ and $e_2 = u_2v_2$.  We apply Menger's theorem
  to $\{u_1, v_1\}$ and $\{u_2, v_2\}$ (or their one-element subsets if these sets are not disjoint).  If the outcome is a pair of
  vertex-disjoint  paths, then we obtain the cycle whose existence is
  claimed.  We may therefore assume that the outcome is a cut vertex
  $x$ that separates $e_1$ from $e_2$.  Hence, $R$ is vertex-wise
  partitioned into $X_1$, $\{x\}$ and $X_2$, in such a way that
  $\{u_1, v_1\} \subseteq X_1 \cup \{x\}$ and
  $\{u_2, v_2\} \subseteq X_2 \cup \{x\}$ and there are no edges
  between $X_1$ and $X_2$.
  We now show that $R[X_1\cup \{ x\} ]$ contains a path from a vertex of degree 1 in $R$
  to $x$ that contains $e_1$. Since $R$ is connected this is clearly true if an endnode of $e_1$
  has degree 1 in $R$. So we may assume that both endnodes of $e_1$ are of degree greater than 1
  in $R$. Let $Y_1$ be the set of all vertices in $X_1$ that have degree 1 in $R$. Note that
  $Y_1\neq \emptyset$ by \ref{i:3} of the definition of the skeleton.
  Suppose $u_1=x$. By \ref{i:3} of the definition of the skeleton, there exists a path in $R[X_1]$ from a vertex of degree 1 to $v_1$, and this path can be extended to a desired path by adding the edge $v_1u_1$. Therefore, by symmetry, we may assume that
  $x\not\in \{ u_1,v_1\}$.
  In $R[X_1 \cup \{x\}]$, we apply Lemma \ref{MengerModified}
  to $\{u_1, v_1\}$ and $\{Y_1,x\}$.
  If we obtain a cut vertex  $y$ that
  separates $\{u_1, v_1\}$ from $\{x\} \cup Y_1$, then $y$ is cut vertex
  of $R$ (separating $e_1$ from $Y_1\cup X_2$) and the component of $R\sm y$
  that contains $e_1$ contradicts~\ref{i:3}.  Hence, we
  obtain two vertex-disjoint  paths, whose union yields a path $P_1$ that contains $e_1$
  from a vertex of degree 1 (in $R$) to $x$.  A similar path $P_2$
  exists in $R[X_2 \cup \{x\}]$. The union of $P_1$ and $P_2$ yields
  the path whose existence is claimed.
\end{proof}

\begin{lemma}\label{l12}
  Let $R$ be a $k$-skeleton.  Every 2-connected induced subgraph $D$ of
  $R$ has at least 3 distinct vertices that have neighbors outside
  $D$. In particular, every 2-connected block of $R$ has at least 3
  attaching vertices.
\end{lemma}

\begin{proof}
  Let $D$ be a 2-connected induced subgraph of $R$. Let
  $u_1$ be a degree 1 vertex of $R$ (it exists by~(i)). Since
  $R$ is connected, there is a path $P_1=u_1\ldots v_1$, where $v_1$
  is the unique vertex of $P_1$ in $D$. In particular, $v_1$ is a
  vertex of $D$ with a neighbor outside of $D$.

  If $v_1$ is not a cut vertex of $R$ that separates
  $P_1 \setminus v_1$ from $D\setminus v_1$, then there is a path
  $P_2=u_1\ldots v_2$, where $v_2$ is the unique vertex of $P_2$ in
  $D$. Otherwise, by~\ref{i:3}, the component $C$ of $R\setminus v_1$
  that contains $D\setminus v_1$ has a vertex $u_2$ of degree 1 in
  $R$, and a path $P_2=u_2\ldots v_2$, where $v_2$ is the unique
  vertex of $P_2$ from $D$. So in both cases we get a vertex $v_2$
  distinct from $v_1$ such that both $v_1$ and $v_2$ have neighbors
  outside $D$.  Since $D$ is 2-connected, $v_1$ and $v_2$ are
  contained in a cycle of $D$, so by \ref{i:5}, $v_1v_2$ is not an
  edge of $R$.

  Suppose that $\{ v_1,v_2\}$ is not a cutset of $R$ that separates
  $(P_1\cup P_2)\setminus \{ v_1,v_2\}$ from a vertex of $D$. Then
  there is a path $P_3=u_3\ldots v_3$ in $R\setminus \{ v_1,v_2\}$,
  where $u_3$ is a vertex of $(P_1\cup P_2)\setminus \{ v_1,v_2\}$ and
  $v_3$ is the unique vertex of $D$ in $P_3$, and hence $v_1,v_2,v_3$
  are the desired three vertices.

So we may assume that $\{ v_1,v_2\}$ is a cutset of $R$ that separates
$(P_1\cup P_2)\setminus \{ v_1,v_2\}$ from a vertex of
$D$. By~\ref{i:2} there is a component $C'$ of
$R\setminus \{ v_1,v_2\}$ such that $C'\cap D\neq \emptyset$ and
$R[C'\cup \{ v_1,v_2\}]$ is not a chordless path. By~\ref{i:4}, $C'$
contains a vertex $u_3$ of degree 1 in $R$, and a path
$P_3=u_3\ldots v_3$, where $v_3$ is the unique vertex of $P_3$ in
$D$. Hence $v_1,v_2,v_3$ are the desired three vertices.

Finally, observe that if $D$ is a block then each of $v_1,v_2,v_3$ is
a cut vertex of $R$, and hence $D$ has at least three attaching
vertices.
\end{proof}

\begin{lemma}\label{l15}
  Let $R$ be a $k$-skeleton.  Let $x_1$ and $x_2$ be  branch vertices
  of $R$ (not necessarily distinct).  Then, there are two paths $P_1 = x_1 \dots y_1$ and
  $P_2 = x_2 \dots y_2$, vertex-disjoint (except at $x_1$ if
  $x_1=x_2$) such that $y_1$ and $y_2$ both have degree 1 and are incident with edges with
  the same label.
\end{lemma}

\begin{proof}
First suppose that there exists a label $i$ that is used at
least twice in $R$, and such that there does not exist a vertex $x$ and
two sets $X, Y \subset V(R)$ such that $X, Y, \{x\}$ form a partition
of $V(R)$, $x_1, x_2 \in X\cup\{x\}$, all degree 1 vertices from edges
with label $i$ are in $Y$, and there are no edges between $X$ and $Y$.
Then, by Menger's theorem there exist two vertex-disjoint  paths (except at $x_1$ if $x_1=x_2$) between $\{x_1,x_2\}$ and the set of all degree 1 vertices from edges with label $i$.

So, suppose that in $R$, for every label $i$ that is
used at least twice in $R$, there exists a vertex $x$ and two sets
$X, Y \subseteq V(R)$ such that $X, Y, \{x\}$ form a partition of
$V(R)$, $x_1, x_2 \in X\cup\{x\}$, all degree 1 vertices from edges
with label $i$ are in $Y$, and there are no edges between $X$ and $Y$.
We then choose $i$, $x$, $X$ and $Y$ subject to the minimality of $X$.
We claim that $x$ is an attaching vertex of $R$.  If $x\in \{x_1, x_2\}$,
it is true by assumption.  Otherwise, if $x$ has a unique neighbor
$x'$ in $X$, then $x'$ is a cut vertex that contradicts the minimality
of $X$ (it separates  $X\sm \{x'\}$ from $Y\cup\{x\}$).  Hence, $x$ has at least
two neighbors in $X$, and at least one in $Y$, so it is indeed an
attaching vertex.

Suppose that $X\cup \{x\}$ contains a limb of
$R$ ending at $x$. This limb cannot have $x_1$ or $x_2$ as its internal vertex,
so we can move it to $Y$ which contradicts the minimality of $X$.  It follows that
$X\cup \{x\}$ is an $x$-petal, or is the union of two $x$-petals $X_1$
(that contains $x_1$) and $X_2$ (that contains $x_2$).  In this last
case, by~\ref{i:10}, there exists a label $j$ that is used in both
$X_1$ and $R\setminus X_1$. So, there exists a path from $x_1$ to an
edge with label $j$ in $X_1\setminus\{x\}$  and a path in $R\setminus X_1$
from $x_2$ to an edge with label $j$, and the conclusion follows.
When $X \cup \{x\}$ is an $x$-petal, we note that there exists another
$x$-petal included in $Y\cup \{x\}$, because $Y\cup \{x\}$ cannot be a
single limb since a label is used twice in $Y$.  Hence, by~\ref{i:10},
there exists a label $j$ that is used in both $X$ and $Y$.  Let $Z$ be
the set of  degree 1 vertices from $X$ which are the degree 1 ends of edges with label~$j$.

First suppose that $x=x_1$. Since $R[Y\cup \{ x\}]$ is connected, it contains a path from $x$
to a vertex incident to an edge labeled $j$. If $x_2=x$ then similarly $R[X\cup \{ x\}]$
contains a path from $x$ to a vertex in $Z$, and the result holds. So we may assume that $x_2\in X$.
By connectivity of $X$ there exists a path in $R[X]$ from $x_2$ to a vertex of $Z$, and the result holds.
Therefore, by symmetry, we may assume that $x\not\in \{ x_1,x_2\}$. Now suppose that
$x_1=x_2$. If there are two paths from $x_1$ to $Z\cup \{ x\}$, then the result holds
(by possibly extending one of the paths from $x$, through $Y$, to a vertex incident to an edge labeled $j$).
Otherwise, by Menger's theorem there is a cut vertex that contradicts the minimality of $X$.
Therefore we may assume that $x_1\neq x_2$.

We now
apply Lemma \ref{MengerModified} to $\{x_1, x_2\}$ and $\{Z,x\}$.  If the
conclusion is two disjoint paths, we are done (by extending the path
ending in $x$ to an edge with label $j$ in $Y$).  And if the outcome
is a cut vertex  $x'$ that separates $\{x_1, x_2\}$ from $\{x\}\cup Z$,
then we define $X'$ as the union of the components of $R\sm \{x'\}$ that
contain $x_1$ and $x_2$.  This contradicts the minimality of $X$.
\end{proof}

\begin{lemma}
  \label{l:bPyr}
  Let $R$ be a $k$-skeleton. Let $P = x_1 \dots x_2$ be a branch of
  $R$ 
  and $x'_1$ a neighbor of $x_1$ not in $P$.
  Then there are
  three paths $P_1 = x_1 \dots y_1$, $P'_1 = x_1 x'_1 \dots y'_1$ and
  $P_2 = x_2\dots y_2$, vertex-disjoint except $P_1$ and $P_1'$ sharing  $x_1$, and such that
  $y_1, y'_1$ and $y_2$ are degree 1 vertices incident with edges with at most two different labels.
\end{lemma}

\begin{proof}
 By Lemma \ref{l15}, there are vertices $y_1$ and $y_2$ of degree~1 incident with edges with the same label, such that there exist vertex-disjoint  paths from $\{x_1,x_2\}$ to $\{y_1,y_2\}$.  We define $X$ as the set of all vertices of degree 1 in $R$, except $y_1$ and $y_2$.
 Note that $X\neq \emptyset$ by \ref{i:1}.
 We apply Lemma \ref{MengerModified} to
  $\{x_1', x_1, x_2\}$ and $\{X,y_1, y_2\}$.  If the output is three
  vertex-disjoint  paths, then the conclusion of the lemma holds ($x_1$
  needs to be added to the path that starts at $x'_1$).
  Otherwise, there exists a cutset $\{a, b\}$ that
  separates $\{x_1', x_1, x_2\}$ from $\{y_1, y_2\}\cup X$.  This
  contradicts~\ref{i:4}.
\end{proof}

\section{Properties of P-graphs}
\label{sec:Pgraphs}

For a P-graph $B$ with special clique $K$ and skeleton $R$, we use the
following additional terminology.  The cliques of $L(R)$ of size at least 3 are called the
{\em big cliques} of $L(R)$. Note that they correspond to sets of
edges in $R$ that are incident to a vertex of degree at least 3.  We
denote by ${\cal K}$ the set that consists of $K$ and all big cliques
of $L(R)$.  Remove from $B$ the edges of cliques in ${\cal K}$. What
remains are vertex-disjoint  paths, except possibly those that meet at
a vertex of $K$. These paths are {\it segments} of $B$; moreover, a segment is an {\em internal segment} if its
endnodes belong to big cliques of $L(R)$, and otherwise it is a {\em
  leaf segment}. If $S$ is a leaf segment and $u\in K$ is an endnode of $S$, we say that $S$ is a
{\em claw segment} if $S$ is not the only segment with endnode $u$; otherwise we say that $S$ is a {\em clique segment}.
Observe that it is possible that a segment is of
length 0, but then it must be an internal segment.  Two segments
$S_1=s_1\ldots t_1$ and $S_2=s_2\ldots t_2$ are {\em parallel} if
$s_1,t_1,s_2,t_2$ are all distinct nodes and for some
$K_1\in {\cal K}\setminus \{  K\}$, $s_1,s_2\in K_1$ and $t_1,t_2\in K$.
Note also that every two cliques of $B$ meet in at most one vertex
(since $R$ is triangle-free).

\begin{lemma}\label{l2}
  Let $B$ be a graph that satisfies all the conditions of being a
  P-graph except that its skeleton fails to satisfies~\ref{i:5} or~\ref{i:8}. Then
  $B$ contains a wheel.
\end{lemma}

\begin{proof}
  Let $R$ be the skeleton of $B$, and $K$ its special clique.

  \noindent{\bf Case 1:} when $R$ fails to satisfy \ref{i:5}.  Suppose
   that in $R$ there exists an edge $e=xy$
  contained in a cycle $C$ such that $x$ and $y$ are both of degree at
  least 3.  If in $R\sm e$, there are two internally vertex-disjoint paths from
  $x$ to $y$, then $R$ contains a cycle with a chord (namely $e$).  So in
  $L(R)$, $e$ is a vertex that is the center of a wheel.  Hence, by
  Menger's theorem, we may assume that in $R\sm e$, there is a
  cut vertex $u$ that separates $x$ from $y$  (note that $u$ is on $C$).  Let $X$ (resp.\ $Y$) be
  the connected component of $R\sm \{e, u\}$ that contains
  $x$ (resp.\ $y$).    We claim that in $R[X\cup \{ u\}]$ there exists a
  path $P_x = x'' \dots x \dots u$ such that $x''$ has degree 1 in
  $R$.

  If $x$ is a cut vertex  of $R$, $P_x$ can be constructed as the union
  of a path from $x$ to $u$ going through the component of $R \sm x$
  that contains $u$, and a path from a vertex $x''$ of degree 1 (that
  exists by~\ref{i:3}) to $x$ going through another component.  So, we
  assume that $x$ is not a cut vertex  of $R$.  Hence, from here on, we
  assume that $R\sm x$ is connected.

  We observe that $\{x, u\}$ is a cutset of $R$, which separates $y$ from each neighbor of $x$ distinct from $y$.  We define $u'$ as
  the vertex of $C\sm e$ closest to $x$ along $C\sm e$ such that
  $\{x, u'\}$ is a cutset of $R$ that separates $y$ from each neighbor of $x$ distinct from $y$.  Let $x'$ be a neighbor of $x$ not in $C$
  (this exists since $x$ has degree at least~3 by assumption). Since $x$ is not a cut vertex of
  $R$, $u'\neq x$.    Let
  $X'$ be the connected component of $R\sm \{x, u'\}$ that contains $x'$. Suppose $xu'\in E(R)$.
  Since $R\setminus x$ is connected there is a path $P$ from $x'$ to $u'$ in $X'\cup\{u'\}$.
  Together with $C$ this provides a cycle with a chord (namely $xu'$), which yields a wheel in $B$. So, $xu'\not\in E(R)$. Let $X_c$ be the connected component of $R\sm \{x, u'\}$
  that contains the vertices from $C\sm e$ that are between $x$ and
  $u'$ (possibly, $X' = X_c$). Note that the vertices of $C\setminus e$ that are
  between $u'$ and $u$  are in the same connected component of $R\sm \{x, u'\}$ as $y$, so none of them is in $X'\cup X_c$. In $R[X' \cup X_c \cup \{x, u'\}]$
  there are two internally vertex-disjoint paths $Q_1$ and $Q_2$ from $x$ to $u'$, for otherwise, by Menger's theorem, a vertex $u''$ from $C$ separates them, and $\{x, u''\}$ is a cutset that
  contradicts $u'$ being closest to $x$ ($\{x,u''\}$ also separates $y$ from each neighbor of $x$ distinct from $y$).  Note that
  $X'\cup \{u', x\}$ or $X_c \cup \{u', x\}$ is not a chordless path, since otherwise
  they induce parallel branches contradicting~\ref{i:2}. Therefore, by \ref{i:4} one of them contains a vertex $x''$ of degree~1.  So,  there
  exists a cycle $C'$ (made of $Q_1$ and $Q_2$) in $R[X' \cup X_c \cup \{x, u'\}]$, and
  a minimal path in $R[X'\cup X_c]$ from $x''$ to a vertex in $C'$.
  This proves that a path visiting in order $x''$, $x$ and $u'$
  exists. We build $P_x$ by extending this path to $u$ along $C\setminus e$.

  We can build a similar path $P_y$.  In $B$, the paths $P_x$ and
  $P_y$ can be  completed to a wheel via $K$ ($e$ is the
  center of this wheel).

  \noindent{\bf Case 2:} when $R$ fails to satisfy~\ref{i:8}.  Suppose
  for a contradiction that some edge $xx'$ of $R$ has label~1, where
  $x$ has degree at least~3 and $x'$ degree 1.  Suppose moreover that
  another edge of $R$, say $yy'$ where $y'$ has degree 1, also receives
  label 1.  Let $Z$ be set of all degree~1 vertices of $R$, except
  $x'$ and $y'$.  We claim that in $R$, there exist two vertex-disjoint
  paths $P_y = x'' \dots y'$ and $P_z = x''' \dots z$, where $z\in Z$ and $x'',x'''$ are some neighbors of $x$ different
  from $x'$.  For otherwise, by Lemma \ref{MengerModified}, there exists a cut vertex $u$ in $R$ that separates $\{x'',x'''\}$ from
  $Z \cup \{y'\}$. Then $\{ u,x'\}$ is a cutset of $R$ such that the connected component
  $C$ of $R\setminus \{ u,x'\}$ that contains $x$ fails to satisfy~\ref{i:4}. Additionally, we may assume that $P_y$ (resp.\ $P_z$) does not contain $x$, since otherwise instead of $P_y$ (resp.\ $P_z$)
  we can take the subpath of $P_y$ (resp.\ $P_z$) from a neighbor of $x$ (on this path) to $y'$ (resp.\ $z$). Now, in $B$, the two paths $P_y$ and $P_z$ together with $x$ and vertices from $K$
  yield a hole, that is the rim of a wheel centered at the vertex
  $xx'$ of $L(R)$.
\end{proof}

\begin{lemma}\label{l3}
  Every P-graph is (theta, wheel, diamond)-free.
\end{lemma}

\begin{proof}
  Let $B$ be a P-graph with skeleton $R$ and special clique $K$. By
  construction of $B$, none of the vertices of $L(R)$ can be centres
  of claws in $B$. So all centres of claws of $B$ are contained in $K$
  and are therefore pairwise adjacent. It follows that $B$ is
  theta-free.  Since $R$ is triangle-free and pendant vertices of
  $L(R)$ have unique neighbors in $K$, and by \ref{i:8}, $B$ is
  diamond-free.

  Suppose that $B$ contains a wheel $(H,x)$.  If $x\in K$ then some
  neighbor $x_1$ of $x$ in $H$ does not belong to $K$, and hence is a
  pendant vertex of $L(R)$.  It follows that the neighborhood of $x_1$ in $L(R)$ is a clique and that $x_1$ has a unique neighbor in $K$. But this contradicts the
  assumption that $x_1$ belongs to the hole $H$ of $B\setminus
  x$. Therefore, $x\not\in K$.

  Since $x$ is a vertex of $L(R)$, it cannot be a center of a claw in
  $B$. Since $B$ is diamond-free, $x$ has neighbors $x_1,x_2,x_3$ in
  $H$, where $x_2x_3$ is an edge and $x_1x_2$ and $x_1x_3$ are
  not. Let $x_1'$ and $x_1''$ be the neighbors of $x_1$ in $H$. Note
  that $x$ has no neighbor in $H\setminus \{ x_1,x_1',x_1'',x_2,x_3\}$
  and it is adjacent to at most one vertex of $\{ x_1',x_1''\}$.

Suppose $x_1\in K$. Then w.l.o.g.\ $x_1'\not\in K$. But then $x_1'$ and $x$ are pendant vertices of $L(R)$ that have
the same labels. Since $\{ x,x_2,x_3\}$ induce a triangle in $L(R)$, $x$ corresponds to a pendant edge of $R$
whose one endnode is of degree at least 3, contradicting \ref{i:8}.
Therefore $x_1\not\in K$, and hence it cannot be a center of a claw. Without loss of generality it follows that the neighbors of $x$ in $H$
are $x_1',x_1,x_2,x_3$ and none of them is in $K$. In particular, $x$ is not a pendant vertex of $L(R)$.

Let $e_x$ be the edge of $R$ that corresponds to vertex $x$ of $L(R)$. Note that the endnodes of $e_x$ are of
degree at least 3 in $R$. So by \ref{i:5}, $e_x$ cannot be contained in a 2-connected block of $R$.
It follows that $x$ is a cut vertex of $L(R)$. Let $C_1$ and $C_2$ be connected components of $L(R)\setminus x$.
Then w.l.o.g.\ $x_1',x_1\in C_1$ and $x_2,x_3\in C_2$, and every path in $B\setminus x$ from $\{ x_1',x_1\}$ to
$\{ x_2,x_3\}$ must go through $K$. It follows that $H$ must have a chord, a contradiction.
\end{proof}

\begin{lemma}
  \label{l:clawHole}
  If $B$ is a P-graph with special clique $K = \{v_1, \dots, v_k\}$
  and $v$ a vertex of an internal segment of $B$, then there exists a
  hole $H$ in $B$ that contains $v$, some vertex $v_i\in K$ and two
  neighbors of $v_i$ in $B\sm K$.
\end{lemma}

\begin{proof}
  We view $v$ as an edge of the skeleton $R$ of $B$.  The edge $v$
  belongs to a branch of $R$ with ends $x_1$ and $x_2$.  Let
  $P_1 = x_1\dots y_1$ and $P_2 = x_2 \dots y_2$ be the two paths
  whose existence is proved in Lemma~\ref{l15} applied to $x_1$ and
  $x_2$.  Let $i$ be the label of edges incident to $y_1$ and $y_2$.  The hole whose
  existence is claimed is induced by $v_i$ and the line graph of the union of  $P_1$, $P_2$, and the
  branch of $R$ from $x_1$ to $x_2$.
\end{proof}

\begin{lemma}
  \label{l:clawCliqueHole}
  Let $B$ be a P-graph with special clique $K = \{v_1, \dots, v_k\}$.
  Let $K_1, K_2, K_3 \in {\cal K}\sm \{ K\}$ be three distinct big cliques.
  Then there exist three paths $P_1 = v \dots u_1$,
  $P_2 = v \dots u_2$ and $P_3 = v \dots u_3$, vertex-disjoint except
  at $v$, with no edges between them (except at $v$), such that
  $v\in K$ and for $i\in \{ 1,2,3\}$, $K_i\cap P_i=\{ u_i\}$.
\end{lemma}

\begin{proof}
 Each of the cliques $K_1, K_2$ and $K_3$
  is a set of edges from $R$ that share a common vertex.  This defines
  three branch vertices $x_1$, $x_2$ and $x_3$ in $R$. By Lemma \ref{l15} there are vertex-disjoint  paths from $\{x_1,x_2\}$ to $\{y_1,y_2\}$, where $y_1$ and $y_2$ are two vertices of $R$ incident with edges that have the same label say 1.   We denote by $X$ the set of
  all the vertices of degree 1 from $R$ different from $y_1$ and $y_2$ ($X$ is
  not empty by (i)).    We now apply Lemma \ref{MengerModified} to $\{x_1, x_2, x_3\}$ and $\{X,y_1, y_2\}$.  If three
  vertex-disjoint paths exist (up to a permutation, say
  $Q_1= x_1 \dots y_1$, $Q_2= x_2 \dots y_2$ and $Q_3 = x_3 \dots y_3$,
  where $y_3 \in X$ and w.l.o.g.\ $y_3$ has label 1 or 2), then we are
  done.  Indeed, in $L(R)$, this yields three chordless paths with no
  edges between them, ending at three vertices with labels 1, 1, 1 or
  1, 1, 2.   By adding $v_1$ or $v_1, v_2$, we obtain the three paths
  whose existence is claimed.

  We may therefore assume that the outcome of Lemma \ref{MengerModified} is a
  set $C$ of at most two vertices that separates $\{x_1, x_2, x_3\}$
  and $X\cup\{y_1, y_2\}$.   This contradicts~\ref{i:3} or \ref{i:4}.
\end{proof}

\begin{lemma}
  \label{l:clawSC}
  Let $B$ be a P-graph with special clique
  $K = \{v_1, \dots, v_k\}$.  Let $S$ be a leaf segment of $B$, whose
  ends are in $K$ and in $K_2 \in {\cal K}\sm \{ K\} $. Let $K_1 \neq K_2$
  be a clique in ${\cal K}\sm \{ K\}$.  Then there exist three paths
  $P_1 = v \dots u_1$, $P_2 = v \dots u_2$ and $P_S = v \dots u_S$,
  vertex-disjoint except at $v$, with no edges between them (except at
  $v$ and for one edge in $K_2$), such that $v\in K$,
  $u_S$ is the endnode of $S$ in $K_2$, and for $i\in \{ 1,2\}$,
  $K_i \cap P_i=\{ u_i\}$.
    Moreover, $P_S =S$ or
  $P_S\sm v = S$.
\end{lemma}

\begin{proof}
  In skeleton $R$ of $B$, the segment $S$ corresponds to limb with
  a pendant edge $e_S$.  Each of
  the cliques $K_1$ and $K_2$ is a set of edges from $R$ that share a
  common vertex.  This defines two vertices $x_1$ and $x_2$ in $R$.

  We suppose first that $e_S$ has a label that is used only once in
  the skeleton $R$.   We apply Lemma~\ref{l15} to $x_1$ and $x_2$.
  This yields paths $P_1$ and $P_2$ that have pendant edges with the same label, say 1.
  Then $S$, line graphs of $P_1$ and $P_2$ and vertex $v_1$, give the desired three paths.

  We now suppose that the label of $e_s$, say 1, is used for another  pendant edge
  with a vertex $y$ of degree 1.  We denote by $X$ the set of all
  degree 1 vertices of $R$, except $y$ and the end of $e_s$.  We
  apply Lemma \ref{MengerModified} to $\{x_1, x_2\}$ and $\{X,y\}$.  If two
  paths are obtained, note that they do not intersect $S$ (because $S$
  is a limb), so by adding $S$ to corresponding paths in $B$, we obtain the paths that we need.
  Otherwise, we obtain a cut vertex, that together with any vertex of
  $S$ yields a cutset of size 2 that contradicts~\ref{i:4}.
\end{proof}

\begin{lemma}
  \label{l:conn4segm}
  Let $B$ be a P-graph with  special clique
  $K = \{v_1, \dots, v_k\}$ such that $k\geq 2$.  Let $S_1$
  and $S_2$ be leaf segments of $B$ that have a common endnode
  $v_i$ in $K$, and let their other endnodes be in $K_1$ and $K_2$,
  respectively ($K_1\neq K_2$). Then there exist  paths $P_1=u'\ldots u_1$
  and $P_2=u''\ldots u_2$, vertex-disjoint  except maybe at a vertex
  of $K$ (when $u'=u''$) and with no edges between them (except for one edge of $K$
  if $u'\neq u''$, or for edges incident to $u'$ when $u'=u''$),
  such that for $i\in \{ 1,2\}$, $K_i\cap P_i=\{ u_i\}$, $u',u''\in K\setminus\{v_i\}$
  and $v_i\not\in P_1\cup P_2$.
\end{lemma}

\begin{proof}
  In skeleton $R$ of $B$, the segments $S_1$ and $S_2$ correspond to limbs with
  pendant edges $e_1$ and $e_2$, respectively.  Each of
  the cliques $K_1$ and $K_2$ is a set of edges from $R$ that share a
  common vertex.  This defines two vertices $x_1$ and $x_2$ in~$R$.

  The label of $e_1$ and $e_2$ is $i$.  We denote by $X$ the set of all
  degree 1 vertices of $R$ that are incident with an edge not labeled with $i$.  We
  apply Menger's theorem to $\{x_1, x_2\}$ and $X$ (by \ref{i:7} and \ref{i:11} we have $|X|\geq 2$).  If two
  paths are obtained, then we are done. Otherwise, we obtain a cut vertex $x$, that separates $\{x_1,x_2\}$ from $X$. Since $x_1$ and $x_2$ are of degree 3 we may assume that $x$ is an attaching vertex, which contradicts~\ref{i:10}.
\end{proof}

\begin{lemma}
  \label{l:conn3segm+1vertex}
   Let $B$ be a P-graph with  special clique
  $K = \{v_1, \dots, v_k\}$ such that $k\geq 2$.
  Let $S_1$ be a leaf segment with endnode $v_i\in K$, and an endnode in $K_1\in \mathcal K\setminus \{ K\}$,
  and let
  $K_2\in \mathcal K\setminus \{ K,K_1\}$.
  Then there exist   paths
  $P_1=u_1\ldots u'$ and $P_2=u_2\ldots u''$ vertex-disjoint
  except maybe at a vertex of $K$ (when $u'=u''$) and with no edges between
  them (except for one edge of $K$ if $u'\neq u''$, and for edges incident to $u'$ when $u'=u''$), such that
  $u',u''\in K\setminus\{v_i\}$,  $v_i\not\in P_1\cup P_2$,  $P_1\cap K_1=\{ u_1\}$ and $P_2\cap K_2=\{ u_2\}$.
\end{lemma}

\begin{proof}
  In skeleton $R$ of $B$, the segment $S_1$ corresponds to a limb with
  pendant edge $e_1$.
  Each of
  the cliques $K_1$ and $K_2$ is a set of edges from $R$ that share a
  common vertex.  This defines two vertices $x_1$ and $x_2$ in $R$.

  The label of $e_1$ is $i$.  We denote by $X$ the set of all
  degree 1 vertices of $R$ that are incident with an edge not labeled with $i$.  We
  apply Menger's theorem to $\{x_1, x_2\}$ and $X$ (by \ref{i:7} and (xi) we have $|X|\geq 2$).  If two
  paths are obtained, then we are done. Otherwise, we obtain a cut vertex $x$, that separates $\{x_1,x_2\}$ from $X$. Since $x_1$ and $x_2$ are of degree 3 we may assume that $x$ is an attaching vertex, which contradicts~\ref{i:10}.
\end{proof}

\begin{lemma}\label{l:pyramidSv}
Let $B$ be a P-graph with special clique $K=\{v_1\}$. If $S$ is a leaf segment of $B$ and $S'$  an
internal segment of $B$, with an endnode in $K'\in\mathcal K$ such that $S\cap K'=\emptyset$, then there exists a pyramid $\Pi$ contained in $B$, such that $S$ and $S'$ are contained in different paths of $\Pi$ and $|\Pi\cap K'|=2$.
\end{lemma}

\begin{proof}
Let $R$ be the skeleton of $B$.
Let $P_S$ (resp. $P_{S'}$) be the limb (resp. branch) of $R$ that corresponds to $S$ (resp. $S'$).
Let $x$ be the degree 1 vertex of $P_S$,  let $x_1$ be the other endnode of $P_S$, and let $y_1$ and $y_2$ be the endnodes of $P_{S'}$, such that edges incident to $y_1$ correspond to nodes of $K'$. Then $x_1\neq y_1$. Furthermore, let $X$ be the set of all degree 1 vertices of $R$ different from $x$.

If in $R$ there exists a vertex $z$ that separates $\{y_1,y_2\}$ from $X$, then for any internal vertex $z'$ of $P_S$ (it exists by \ref{i:7} and \ref{i:8}),
the set $\{z,z'\}$ is a cutset of $R$ that contradicts (iv). So, by Menger's theorem there are vertex-disjoint paths
$P'=y_1\ldots x'$ and $P''=y_2\ldots x''$, where $x',x''\in X$. Suppose that in $R\setminus y_1$ there exists a path from $x$ to
$(P'\cup P'')\setminus\{y_1\}$, and let $P'''$ be chosen such that it has the minimum length. Then $L(P'\cup P''\cup P'''\cup P_{S'})\cup \{v_1\}$ induces the desired pyramid.

So, we may assume that $y_1$ is a cut vertex of $R$, such that $x$ and $(P'\cup P'')\setminus\{y_1\}$ are contained in
different connected components of $R\setminus y_1$. Let $C_x$ be the connected component of $R\setminus\{y_1\}$ that contains
$x$, let $e_x$ be the edge incident to $x$ and let $e_y$ be an edge of $P_{S'}$. By Lemma \ref{l4} there exists a path $P$ in $R$
that contains edges $e_x$ and $e_y$ whose endnodes  are of degree 1 in $R$.  Note that $P$ contains $P_{S'}$. Let $x_1'$ be a node adjacent to $x_1$ that does not
belong to $P$. Since $x_1\neq y_1$, we have $\{x_1,x_1'\}\subseteq C_x$.
Let us apply Lemma \ref{MengerModified} in graph $C_x$ to $\{x_1,x_1'\}$ and $\{X_1,x\}$, where $X_1$ is the set of all degree 1
(in $R$) nodes of $C_x$ different from $x$ ($X_1$ is non-empty, since otherwise for any internal vertex $z'$ of $P_S$ the set
$\{z',y_1\}$ is a cutset of $R$ that contradicts  \ref{i:4}). If vertex-disjoint  paths $P_1=P_S$ and $P_1'$ are obtained, then
$L(P\cup P_1')$ and $v_1$ induce a desired pyramid $\Pi$. Otherwise, let $z$ be a vertex of $C_x$ that separates $\{x_1,x_1'\}$ from
$X_1\cup\{x\}$. But then $\{z,y_1\}$ is a cutset of $R$ that contradicts \ref{i:4}.
\end{proof}

\begin{lemma}\label{l:pyramidSS}
Let $B$ be a P-graph with special clique $K=\{v_1\}$. If $S_1$ and $S_2$ are leaf segments of $B$,
then there exists a pyramid $\Pi$ contained in $B$, such that $S_1$ and $S_2$ are contained in different paths of $\Pi$.
\end{lemma}

\begin{proof}
Let $x_1$ (resp.\ $x_2$) be degree 1 vertex of skeleton $R$ of $B$ incident to pendant edge that corresponds to a vertex of $S_1$ (resp.\ $S_2$).
Furthermore, let  $X$ be the set of all degree 1 vertices of $R$ different from $x_1$ and $x_2$.
Note that by \ref{i:1}, $X\neq \emptyset$.
Let $P'$ be a direct connection from
$\{x_1,x_2\}$ to $X$ in $R$, and w.l.o.g.\ let $x_1$ be the neighbor of one endnode of $P'$. Let $P''$ be a direct connection from $x_2$ to $P'$.
Then $L(P'\cup P'')\cup \{v_1\}$ induces the desired pyramid.
\end{proof}

\begin{lemma}\label{newsec4l}
Let $B$ be a P-graph with special clique $K=\{v_1, \ldots ,v_k\}$.
Let $v$ be the vertex of an internal segment of length 0, let $K_1\in {\cal K}\setminus \{ K\}$ be such that $v\in K_1$
and let $u\in K_1\setminus \{ v \}$.
Then $B$ contains a pyramid $\Pi =3PC(uvx,y)$ such that $x\in K_1$ and $y\in K$.
\end{lemma}

\begin{proof}
Let $R$ be the skeleton of $B$, and let $e=x_1x_2$ be an edge of $R$ that corresponds to vertex $v$.
Let $x_1'$ be the neighbor of $x_1$ in $R$ such that $x_1x_1'$ corresponds to vertex $u$.
Let $P_1=x_1\ldots y_1$, $P_1'=x_1x_1'\ldots y_1'$ and $P_2=x_2\ldots y_2$ be the three paths obtained by
applying Lemma \ref{l:bPyr} to $x_1,x_1'$ and $x_2$. Then $y_1,y_1'$ and $y_2$ are vertices of degree 1 in $R$
incident with edges with at most two different labels, say $i$ and $j$. It follows that $L(\{x_1,x_2\}\cup P_1\cup P_1'\cup P_2)$
and $\{ v_i,v_j\}$ induce the desired pyramid in $B$.
\end{proof}

\section{Attachments to a P-graph}
\label{sec:attach}

\begin{lemma}[\cite{twf-p1}]\label{l:Hv}
  Let $G$ be a (theta, wheel)-free graph. If $H$ is a hole of $G$ and
  $v$ a node of $G\setminus H$, then the attachment of $v$ over
  $H$ is a clique of size at most 2.
\end{lemma}

\begin{lemma}\label{l:2inAhole}
  In a P-graph $B$ every pair of segments is contained in a hole.
  Also, every pair of vertices of $B$ is contained in a hole.
\end{lemma}

\begin{proof}
  Follows directly from Lemma~\ref{l4} (note that every vertex of $B$
  is contained in a segment of $B$, and every segment contains a
  vertex that corresponds to an edge of skeleton $R$ of $B$).
\end{proof}

\begin{lemma}\label{node-attach}
  Let $G$ be a (theta, wheel, diamond)-free graph and $B$ a P-graph
  contained in $G$. If $v \in G\setminus B$, then either
  $|N_B (v)|\leq 1$ or $N_B(v)$ is a maximal clique of $B$.
\end{lemma}

\begin{proof}
  Since $G$ is diamond-free, it suffices to show that $N_B(v)$ is a
  clique.  Assume not and let $v_1$ and $v_2$ be non adjacent
  neighbors of $v$ in $B$.    By Lemma \ref{l:2inAhole}, $v_1$
  and $v_2$ are contained in a hole $H$ of $B$. But then $H$ and $v$
  contradict Lemma \ref{l:Hv}.
\end{proof}

Let $G$ be a (theta, wheel, diamond)-free graph and
$\Pi =3PC(x_1x_2x_3,y)$ be a pyramid contained in $G$.
Then $\Pi$ is a long pyramid and by Lemma \ref{pyrBas} it is a P-graph
with special clique $\{ y\}$.
 For
$i=1, 2, 3$, we denote by $S_i$ the branch of $\Pi$ from $y$ to $x_i$
and we denote by $y_i$ the neighbor of $y$ on this path.  By Lemma~\ref{node-attach} it follows that the
attachment of a node $v\in G\setminus \Pi$ over $\Pi$ is a clique of
size at most~3. For $i=1,2,3$, we shall say that $v$ is of {\em Type i
  w.r.t.\ $\Pi$} if $|N_{\Pi}(v)|=i$.  We now define several kinds of
paths that interact with $\Pi$.

\begin{itemize}

\item A {\it crossing} of $\Pi$ is a chordless path
  $P=p_1\ldots p_k$ in $G\setminus \Pi$ of length at least 1, such
  that $p_1$ and $p_k$ are of Type 1 or 2 w.r.t.\ $\Pi $, for some
  $i,j\in\{1,2,3\}$, $i\neq j$, $N_{\Pi }(p_1)\subseteq S_i$,
  $N_{\Pi }(p_k)\subseteq S_j$, $p_1$ has a neighbor in
  $S_i\setminus\{y\}$, $p_k$ has a neighbor in $S_j\setminus\{y\}$, at
  least one of $p_1,p_k$ has a neighbor in
  $(S_i\cup S_j) \setminus \{ x_i,x_j\}$ and no node of
  $P\setminus\{p_1,p_k\}$ has a neighbor in $\Pi $.

\item Let $P=p_1\ldots p_k$ be a crossing of $\Pi $ such that for
  some $i,j\in\{1,2,3\}$, $i\neq j$, $N_{\Pi}(p_1)  =\{ y_i\}$ or
  $\{ y_i,y\}$, $p_k$ is of Type 2 w.r.t.\ $\Pi $ and
  $N_{\Pi }(p_k)\subseteq S_j\setminus\{y, y_j\}$.  Moreover, if
  $N_{\Pi }(p_1) =\{ y_i\}$ then $S_i$ has length at least 3.  Then
  we say that $P$ is a {\it crosspath} of $\Pi $ (from $y_i$ to
  $S_j$).  We also say that $P$ is a {\em $y_i$-crosspath} of
  $\Pi $.

\item If $P=p_1\ldots p_k$ is a crossing of $\Pi $ such that $p_1$
  and $p_k$ are of Type 2 w.r.t.\ $\Pi $ and neither is adjacent to
  $\{ y, y_1, y_2, y_3, x_1, x_2, x_3\}$, then $P$ is a {\it
    loose crossing} of $\Pi $.
\end{itemize}

A long pyramid with a loose crossing is a P-graph. To see this,
consider a 1-skeleton made of a chordless cycle $C$ together with
three chordless paths $P_1,P_2,P_3$, all of length at least 2,
such that for $i\in \{ 1,2,3\}$, $P_i\cap C=\{ v_i\}$, and $v_1,v_2,v_3$
are pairwise distinct and nonadjacent.
  The three pendant edges of the paths
receive label~1, and the special clique has size 1.

A long pyramid with a crosspath is also a P-graph.  The special clique
$K$ is $\{y_i, y\}$ (when $N_{\Pi }(p_1) =\{ y_i\}$) or
$\{y_i, y, p_1\}$ (when $N_{\Pi }(p_1) =\{ y_i, y\}$), so it has
size 2 or 3.  It is easy to check that removing $K$ yields the line
graph of a tree that has two vertices of degree~3 and four pendant
edges that receive labels 1, 1, 2, 2 when $|K|=2$ and 1, 1, 2, 3, when
$|K|=3$.

\begin{lemma}\label{crossings}
  Let $G$ be a (theta, wheel, diamond)-free graph.  If
  $P=p_1\ldots p_k$ is a crossing of a $\Pi =3PC(x_1x_2x_3,y)$
  contained in $G$, then $P$ is a crosspath or a loose crossing of
  $\Pi $.
\end{lemma}

\begin{proof}
   Assume w.l.o.g.\ that $p_1$ has a neighbor in $S_1\setminus \{ y\}$,
  and $p_k$ in $S_2\setminus \{ y\}$.  Not both $p_1$ and $p_k$ can be
  adjacent to $y$, since otherwise $N_{\Pi}(p_1)=\{ y_1,y\}$ and
  $N_{\Pi}(p_k) =\{ y_2,y\}$, and hence $S_1\cup S_2\cup P$
  induces a wheel with center $y$.  Suppose that both $p_1$ and $p_k$
  are of Type 2 w.r.t.\ $\Pi $.
  If $p_1$ is adjacent to $y$, then $P$ is a crosspath, since otherwise
  $p_k$ is adjacent to $y_2$ and not to $y$, and hence
  $S_2\cup S_3\cup P$
  induces a wheel with center $y_2$. So we may assume that neither $p_1$
  nor $p_k$ is adjacent to $y$. If $p_1$ is adjacent to $y_1$, then
  $G[(\Pi  \setminus \{x_2\})\cup P]$ contains a wheel with center $y_1$.
  So $p_1$ is not adjacent to $y_1$, and by symmetry
  $p_k$ is not adjacent to $y_2$.
  If $p_1$ is adjacent to $x_1$, then
  $G[(\Pi  \setminus \{y_2\})\cup P]$ contains a wheel with center $x_1$.
  So $p_1$ is not adjacent to $x_1$, and by symmetry $p_k$ is not adjacent to $x_2$.
  It follows that $P$ is a loose crossing.

  Without loss of generality we may now assume that $p_1$ is of Type 1 w.r.t.\ $\Pi $.
  If $p_k$ is also of Type 1, then  $S_1\cup S_2\cup P$ induces a theta. So
  $p_k$ is of Type 2. If $p_1$ is
  not adjacent to $y_1$, then
  $G[(\Pi \setminus \{ x_2\} )\cup P]$ contains a $3PC(y,\widetilde{y})$, where
  $\widetilde{y}$ is the only neighbor of $p_1$ on $\Pi$.
  So $p_1$ is adjacent to $y_1$. Since $S_1\cup S_2\cup P$ cannot induce a wheel
  with center $y$, $p_k$ is not adjacent to $y$.
  Since $S_2\cup S_3\cup P\cup \{  y_1\}$ cannot induce a wheel with center $y_2$,
  $N_{\Pi } (p_k)\subseteq S_2\setminus \{ y,y_2\}$.
  If $S_1$ is of length 2, then $G[(\Pi  \setminus \{ y_2\} )\cup P]$
  contains a wheel with center $x_1$.
  Therefore $S_1$ is of length at least 3, and hence $P$ is a crosspath.
\end{proof}

\begin{lemma}
  Let $G$ be a (theta, wheel, diamond)-free graph.  If $G$ contains a
  pyramid $\Pi $ with a crossing $P$, then $G[\Pi  \cup P]$ is a P-graph.
\end{lemma}

\begin{proof}
  Follows from Lemma~\ref{crossings} and the fact already mentioned
  that a pyramid together with a loose crossing or a crosspath is a
  P-graph.
\end{proof}

Let $S$ be a segment of a P-graph $B$ such that its endnodes are in $K_1$ and $K_2$. Then we say that $S\cup K_1\cup K_2$ is an {\it extended segment} of $B$.

\begin{lemma}\label{l:extendB}
Let $B$ be a P-graph with  special clique $K$ which is contained in a (theta,wheel,diamond)-free graph $G$. Let $P=u\ldots v$ be a path in $G\setminus B$
whose interior nodes have no neighbors in $B$ and one of the following holds:

\begin{itemize}
  \item[(1)] $N_B(u)$ and $N_B(v)$ are cliques of size at least 2 in $B\setminus K$ which are not contained in the same extended segment of $B$.

  \item[(2)] $N_B(u)=K$, where $|K|\geq 2$, and $N_B(v)$ is a clique of size at least 2 which is in $B\setminus K$, but not in an extended clique segment of $B$.

  \item[(3)] $N_B(u)=\{w\}\subseteq K$, and $N_B(v)$ is a clique of size at least 2 in $B\setminus K$ which is not in a extended segment of $B$ incident with $w$.
\end{itemize}
Then $G[B\cup P]$ is a P-graph contained in $G$.
\end{lemma}

\begin{proof}
Let $K=\{v_1,\ldots,v_k\}$ and let $R$ be the skeleton of $B$.
In all three cases neighbors of $v$ in $B$ are in fact in $L(R)$, and they correspond to some edges of $R$ all incident to a single vertex  $k_2\in R$.
By Lemma \ref{node-attach}, $v$ is adjacent to all vertices that correspond to edges incident to $k_2$. We now consider each of the cases.

\vspace{2ex}

\noindent(1) Let $k_1$ be the vertex of $R$ whose incident edges correspond to vertices of the clique $N_B(u)$ in $L(R)$.
Note that by Lemma \ref{node-attach}, $u$ is adjacent to all vertices that correspond to edges incident to $k_1$.
Construct graph $R'$ from $R$ by adding a
branch $P_R$ between $k_1$ and $k_2$, of length one more than the
length of $P$. We prove that $R'$ is a $k$-skeleton.

By Lemma \ref{l2}, it suffices to check that all conditions other than \ref{i:5} and~\ref{i:8} are met.  Since $P$ is of length
at least 1, $P_R$ is of length at least~2, and thus~(i) holds.  Since $N_B(u)$ and $N_B(v)$ are not contained in the same extended segment of $B$,  no
branch of $R$ contains both $k_1$ and $k_2$, and hence~\ref{i:2} holds.

Note that $R$ and $R'$ have the same degree 1 vertices and the same
limbs.  It follows that~\ref{i:6}, \ref{i:7}, \ref{i:9} and \ref{i:11} hold for
$R'$.

Let $x$ be a cut vertex of $R'$.  Since $R$ is connected, $x$ is not an
internal vertex of $P_R$.  Hence, $x$ is also a cut vertex of $R$ and
every component of $R'\sm x$ contains a union of components of $R\sm x$.  It
follows that~\ref{i:3} holds.  Also, every $x$-petal of $R'$ is a union of some $x$-petals of $R$ and some vertices of $P_R$, and therefore \ref{i:10} holds.

To prove~\ref{i:4} let $\{a, b\}$ be a cutset of $R'$.  If $a$ and $b$
are in the interior of $P_R$, one component of $R'\sm \{a, b\}$ is a
chordless path from $a$ to $b$, and the other contains all the
vertices of $R'$ of degree 1, so~\ref{i:4} holds.  If one of $a$ or
$b$, say $a$, is in the interior of $P_R$, and the other (so, $b$) is
not, then $b$ is a  cut vertex of $R$. Also, every component of $R'\sm
\{a, b\}$ contains a component of $R\sm b$.  Hence~\ref{i:4} holds
because~\ref{i:3} holds for $R$.  Finally, if none of $a$ and $b$ is
in the interior of $P_R$, then $\{a, b\}$ is also a cutset of $R$, and
every components of $R'\sm \{a, b\}$ contains a component of $R\sm \{a,
b\}$.  Therefore,~\ref{i:4} holds for $R'$ because it holds for
$R$. Thus \ref{i:4} holds, and our claim is proven.

\vspace{2ex}

\noindent(2)   Construct graph $R'$ from $R$ by adding a chordless
path $P_R$ of the same length as $P$, whose one endnode is $k_2$ and
the remaining nodes are new.  Note that pendant edges of $R$ are also
pendant edges of $R'$, and $R'$ has one new pendant edge (the one
incident to the vertex of degree 1 in $R'$ that is in $P_R$). Let us assign
label $k+1$ to the new pendant edge. We claim
that $R'$ is a skeleton.  By Lemma \ref{l2}, there is no need
to check~\ref{i:5} and~\ref{i:8}.  Since $P_R$ is a limb,~(i),~\ref{i:2},~\ref{i:6}, \ref{i:7} and \ref{i:11} hold for $R'$ because they hold for $R$
and since in this case $k\geq 2$.

Let us show that~\ref{i:9} holds.  It could be that the limb
that we add to $R$ to build $R'$ is in fact parallel to a limb $Q$
of $R$, that corresponds to a clique segment $S'$ of
$B$. If the label of pendant edge of $Q$ is used only once,
then $N_B(v)$ is contained in an extended  clique segment of $B$ (namely extended segment of $S'$), a contradiction.
So \ref{i:9} holds.

The conditions~\ref{i:3}, \ref{i:4} and~\ref{i:10} hold for $R'$
because they hold for $R$.  Indeed, in $R'$, we added a limb, this
only possibly adds a vertex of degree~1 to a component,  making the
condition easier to satisfy.

\vspace{2ex}

\noindent(3) Let $w=v_i$.  We
build a path $P_R$ of the same length as $P$ and we consider the
graph $R'$ obtained from $R$ by attaching $P_R$ at $k_2$.  Hence, in
$P_R$ there is a pendant edge, and we give it label $i$.  We claim
that $R'$ is a skeleton.  By Lemma \ref{l2}, there is no need
to check~\ref{i:5} and~\ref{i:8}.  Since $P_R$ is a limb,~(i),~\ref{i:2},~\ref{i:6}, \ref{i:7} and \ref{i:11} hold for $R'$ because they hold for $R$.

Condition~\ref{i:9} also holds, since the limb that we add to build $R'$
has pendant edge with label $i$ that is now used at least twice,
and it is not parallel to some other limb with pendant edge $i$
by the condition of the lemma.

The conditions~\ref{i:3}, \ref{i:4} and~\ref{i:10} hold for $R'$
because they hold for $R$.  Indeed, in $R'$ we added a limb, which
only possibly adds a vertex of degree~1 to a component,  making the
condition easier to satisfy.
\end{proof}

\begin{lemma}\label{path-attach}
  Let $G$ be a (theta, wheel, diamond)-free graph, and let $B$ be the
  P-graph contained in $G$ with special clique $K=\{v_1,\ldots,v_k\}$
  and skeleton $R$, such that $k$ is maximum, and among all P-graphs
  contained in $G$ and with special clique of size $k$, $B$ has the
  maximum number of segments.  Let $P=u\ldots v$ be a chordless path
  in $G\setminus B$ such that $u$ and $v$ both have neighbors in $B$
  and no interior node of $P$ has a neighbor in $B$.  Then one of the
  following holds:
\begin{itemize}
\item[(1)] $N_B(P)\subseteq K'$, where $K'\in {\cal K}$.
\item[(2)] There exists a segment $S$ of $B$, of length at least 1, whose endnodes are in
$K_1\cup K_2$ where $K_1,K_2\in {\cal K}$, such that $N_B(P)\subseteq K_1\cup K_2\cup S$.
Moreover, if $u$ (resp. $v$) has a neighbor in $K_i \setminus S$, for some $i\in \{ 1,2\}$,
then $u$ (resp. $v$) is complete to $K_i$.
\end{itemize}
\end{lemma}

\begin{proof}
  Before proving the theorem, note that in the proof, conclusion (2)
  can be replaced by a weaker conclusion :

  \begin{itemize}
  \item[(2')] There exists a segment $S$ of $B$, of length at least 1,
    whose endnodes are in $K_1\cup K_2$ where $K_1,K_2\in {\cal K}$,
    such that $N_B(P)\subseteq K_1\cup K_2\cup S$.
   \end{itemize}

   Indeed, if (2') is satisfied, then (1) or (2) is  satisfied.
   Let us prove this.  Suppose that (2') holds, but neither (1) nor
   (2) does.  Up to symmetry, and by Lemma \ref{node-attach}, this means that $N_B(u)$ is a
   single vertex $u'$ of $K_1 \setminus S$.  If $N_B(v)$ is also a
   single vertex $v'$, then by Lemma~\ref{l:2inAhole}, $P$ together
   with a hole that goes through $u'$ and $v'$ forms a theta (note that since (1) does not hold, $v'\in (S\cup K_2)\setminus K_1$
   and hence since $R$ has no parallel branches by \ref{i:2}, $u'v'$ is not an edge).
   By Lemma \ref{node-attach}, we may therefore assume that $N_B(v)=K_2$ or $N_B(v)$ is a clique of size 2 in $S$.

   We first suppose that $K\not\in \{ K_1,K_2\}$.
   In $R$, $N_B(u)$ is an edge $y_1y_1'$, where $y_1'$ is a branch vertex and $S$ corresponds to a branch
   $P'=y_1'\ldots y_2'$. We apply Lemma \ref{l:bPyr} to $y_1,y_1'$ and $y_2'$.
   Let $P_1,P_2$ and $P_3$ be the three paths obtained and suppose that label $i$ is used on pendant edges of two
   of these paths. Then the graph induced by $L(P_1\cup P_2\cup P_3)$ together with $S\setminus K_1$, $P$ and $K$
   contains a $3PC(u',v_i)$ (note that by \ref{i:8}, $u'v_i$ is not an edge).

   Next suppose that $K_1=K$ and let $u'=v_i$.
   First observe that if $N_B(v)=K_2$ and there exists a segment $S'$ of $B$ with endnode $v_i$ and an endnode in $K_2$,
   then $P$ satisfies (2) w.r.t.\ $S'$. So this cannot happen. It follows that if $N_B(v)\cap K=\emptyset$ then by part (3)
   of Lemma \ref{l:extendB}, the maximality of $B$ is contradicted.
   So let $v_j\in N_B(v)\cap K$, where $v_j\neq v_i$, and let $S'$ be a segment of $B$ with endnode $v_i$.
   Let $Q$ be a direct connection from $S'$ to $S$ in $B\setminus K$. Then $G[S\cup S'\cup P\cup Q]$ is a wheel with center $v_j$,
   a contradiction.

   Therefore $K_2=K$.
   First suppose that $u'$ is a vertex of an internal segment of $B$. Then by Lemma \ref{l:clawHole}, there exists a hole $H$
   that contains $u'$ and a vertex $v_j\in K$ such that neighbors of $v_j$ in $H$ are in $B\setminus K$.
   If $S$ is not contained in $H$, then $G[H\cup P\cup (S\setminus K_1)]$ contains a $3PC(u',v_j)$ (note that since $u'$
   belongs to an internal segment of $B$, $u'v_j$ is not an edge).
   So $S$ is contained in $H$, and hence $v_j$ is an endnode of $S$. If $N_B(v)=K$ then $G[H\cup P]$ is a theta.
   So $N_B(v)\neq K$.
   In $R$, $u'$ is an edge $y_1y_1'$, where $y_1'$ is a branch vertex, and $S$ corresponds to a limb $P'=y_1'\ldots y_2'$.
   Let $X$ be the set of all degree 1 vertices of $R$ incident with pendant edges labeled with $j$ not including $y_2'$
   (note that $X$ is nonempty) and $Y$ the set of all other degree 1 vertices of $R$ not including $y_2'$.
   If in $(R\setminus P') \cup \{ y_1'\}$ there are vertex-disjoint paths $P_1$ and $P_2$ from $\{ y_1',y_1\}$ to $\{ X,Y\}$,
   then $G[L(P_1\cup P_2)\cup P\cup (S\setminus K_1)\cup \{ u'\}]$ contains a $3PC(u',v_j)$.
   So, by Lemma \ref{MengerModified}, there is a vertex $x$ in $R$ that separates $\{ y_1',y_1\}$ from $X\cup Y$ in $(R\setminus P')\cup \{ y_1'\}$,
   and therefore $\{ y_2',x\}$ is a cutset of $R$ that contradicts \ref{i:4}.

   It follows that $u'$ is an endnode of a leaf segment $S'$ of $B$.
   Since (2) does not hold for $P$ and $S'$, $N_B(v)\neq K$ and hence $N_B(v)$ is
   a clique of size 2 in $S$.
   Let $v_i$ (resp. $v_j$) be the endnode of $S$ (resp. $S'$) in $K$. Suppose $v_i=v_j$.
   Then by \ref{i:9}, $R$ has no branches, so by (i), $G[(B\setminus (S\cap K_1)) \cup P]$ contains a
   $3PC(u',v_i)$ (note that $u'v_i$ is not an edge by \ref{i:8}).
   So $v_i\neq v_j$. By \ref{i:9} there is a segment
   $S''\not\in \{ S,S'\}$ with an endnode in $\{ v_i,v_j\}$.
   Note that $S''$ does not have an endnode in $K_1$.
    Let $Q$ be a direct connection from $S''$ to $K_1$
   in $B\setminus K$.
    Then $G[(S\setminus K_1)\cup S'\cup S''\cup P\cup Q]$ either contains a $3PC(u',v_i)$ (if $S''$ has endnode $v_i$)
    or $3PC(u',v_j)$ (if $S''$ has endnode $v_j$, note that in this case by \ref{i:8}, $u'v_j$ is not an edge).
    Therefore, if (2') holds then (1) or (2) holds.

   \vspace{2ex}

   We are now back to the main proof.  Suppose the conclusion of the
   theorem fails to be true.  By
   Lemma~\ref{node-attach}, it suffices to consider the following
   cases.
   \\
   \\
   {\bf Case 1:} For some $K_1, K_2\in {\cal K}\setminus \{ K\}$,
   $N_B(u)=K_1$ and $N_B(v)=K_2$.
   \\
   \\
   Since (1) does not hold, $K_1\neq K_2$. Let us first prove that no segment of $B$ has endnodes in $K_1\cup K_2$.

   Suppose to the contrary that some segment $S$ of $B$ has endnodes
   in $K_1\cup K_2$.  Since (2') does not hold, $S$ is of length 0,
   say $S=x$. So $S$ is an internal segment of $B$.  Let $e_x$ be the
   edge of $R$ that corresponds to $x$.  By Lemma~\ref{l11}, $e_x$ is
   a cut edge of $R$, and hence $x$ is a cut vertex of $L(R)$.  For
   $i=1,2$, let $C_i$ be the connected component of $L(R)\setminus x$
   that contains $K_i\setminus x$. Note that the endnodes of $e_x$ in
   $R$ are cut vertices of $R$, and hence by \ref{i:3}, $C_i$ has a
   pendant vertex, for $i=1,2$.  It follows that $B$ contains a
   chordless $wz$-path $Q$, where $w\in K_1\setminus x$,
   $z\in K_2\setminus x$ and no interior node of $Q$ has a neighbor in
   $K_1\cup K_2$.  But then $P\cup Q\cup \{ x\}$ induces a wheel with
   center $x$. Therefore, no segment of $B$ has an endnode in $K_1\cup K_2$.

    Now, by part (1) of Lemma \ref{l:extendB}, this contradicts the maximality of $B$.

\vspace{2ex}

\noindent {\bf Case 2:} For some $K_1\in {\cal K}\setminus \{ K\}$,
$N_B(u)=K_1$ and $N_B(v)=K$.
\\
\\
Since (2') does not hold,
there is no (leaf) segment with endnodes in $K_1$ and $K$, and so by parts (2) and (3) of Lemma \ref{l:extendB} and maximality of $B$, this case is impossible.

\vspace{2ex}

\noindent {\bf Case 3:} For some segment $S$ of $B$, $N_B(u)=K$ and
$N_B(v)\subseteq S$.
\\
\\
Since (2') does not hold, $S$ is an internal segment of $B$.  Let $v'$
be a neighbor of $v$ in $S$.  Apply Lemma~\ref{l:clawHole} to $B$ and
$v'$.  This provides a hole $H$ in $B$ that contains $v'$ and a single node of $K$.
Note that $H$ contains $S$ because
$S$ is a segment.  If $v'$ is the only neighbor of $v$ in $S$, then
$H$ and $P$ form a theta, a contradiction.  So, by Lemma~\ref{l:Hv},
for some vertex $v''$ of $S$ adjacent to $v'$, $N_B(v)=\{ v',v''\}$. By parts (2) and (3) of Lemma \ref{l:extendB} this contradicts the maximality of $B$.

\vspace{2ex}

\noindent {\bf Case 4:} For some $K_1\in {\cal K}\setminus \{ K\}$ and
some internal segment $S$ of $B$, $N_B(u)=K_1$ and
$N_B(v)\subseteq S$.
\\
\\
Let $K_2$ and $K_3$ be the end cliques of $S$. Since (1) and (2') do
not hold, $K_1\not\in \{ K_2,K_3\}$. We apply
Lemma~\ref{l:clawCliqueHole} to $K_1$, $K_2$ and $K_3$.  This provides
three paths $P_1$, $P_2$ and $P_3$.
If $N_B(v)=\{ v'\}$ then $P_1,P_2,P_3,S$ and $P$  induce a theta. So by Lemma \ref{node-attach}
$N_B(v)=\{ v',v''\}$ where $v'$ and
$v''$ are two adjacent vertices of $S$. By part (1) of Lemma \ref{l:extendB}
this contradicts the maximality of $B$.

\vspace{2ex}

\noindent {\bf Case 5:} For some $K_1\in {\cal K}\setminus \{ K\}$ and
some leaf segment $S$ of $B$, $N_B(u)=K_1$ and $N_B(v)\subseteq S$.
\\
\\
Let the endnodes of $S$ be in cliques $K$ and
$K_2\in {\cal K}\setminus \{ K\}$.  Since $S$ is a leaf segment of
$B$, it is of length at least 1. Since (2') does not hold,
$K_1\neq K_2$.  Let $v'$ be a neighbor of $v$ in $S$, and let
$P_1=w_1\ldots w$, $P_2=w_2\ldots w$ and $P_S=w_s\ldots w$ be paths
obtained when Lemma ~\ref{l:clawSC} is applied to segment $S$ and clique $K_1$.

First, let us assume that
$N_B(v)=\{v'\}$. If $v'\neq w$ and
$v'$ is not adjacent to $w$, then $G[P_1\cup P_2\cup P_S\cup P]$
induces a $3PC(v',w)$, a contradiction.
So, $v'=w$ or $v'w$ is an edge. If
$v'\in K$, then by part (3) of Lemma \ref{l:extendB} and maximality
of $B$, there is a segment $S'$ with one endnode in $K_1$ and the other
$v'$. But then $P$ and $S'$ satisfy condition (2'). So, $v'\not\in K$,
and hence $v'w$ is an edge.
Suppose $k=1$. Let $z$ be a node of $K_1$ that belongs to an internal segment of $B$ (note that since $K_1\neq K_2$, and since
$R$ is connected by \ref{i:1}, it follows that $R$ has a branch and $z$ exists by \ref{i:9}).
By Lemma \ref{l:pyramidSv} there exists a pyramid $\Pi$ contained in $B$ such that $S$ and $z$
belong to different paths of $\Pi$ and $|\Pi\cap K_1|=2$. So, $N_{\Pi}(u)$ is an edge of a path of $\Pi$ that contains $z$.
Note that since $G$ is wheel-free, $\Pi$ is a long pyramid and by Lemma \ref{crossings} $P$ is a crosspath of $\Pi$.
But then $G[\Pi \cup P]$ is a P-graph with special clique of size greater than 1, contradicting our choice of $B$
(since $k=1$). Therefore, $k>1$.
Let $Q_1$ and
$Q_2$ be paths obtained when Lemma \ref{l:conn3segm+1vertex} is applied to
$S$ and $K_1$.
Then $G[Q_1\cup Q_2\cup S\cup P]$ induces a theta or a wheel, a contradiction.
So, by Lemma \ref{node-attach}, $N_B(v)$ is a clique of size 2.

If $N_B(v)\cap K=\emptyset$, then, by (1) of Lemma
\ref{l:extendB}, we have a contradiction to the maximality of $B$. So,
$N_B(v)=\{v',v''\}$, where $v''\in K$. If $v''\neq w$, then $G[P_1\cup P_2\cup P_S\cup P]$
induces a wheel, a contradiction. So, $v''=w$. If $k=1$, then by
Lemma \ref{l:pyramidSv} there exists a pyramid $\Pi$,
contained in $B$, such that $S$ and $z$ are in different paths of $\Pi$,
where $z$ is
a node of $K_1$ that belongs to an internal segment of $B$ (it exists by the same argument as in the previous paragraph).
Note that $w$ is the center of the claw of $\Pi$.
But then
$G[\Pi\cup P]$ is a P-graph whose special clique is of size 3, contradicting our choice of $B$.
So $k>1$.
Let $Q_1$ and
$Q_2$ be paths obtained when Lemma \ref{l:conn3segm+1vertex} is applied to
$S$ and a node $z\in K_1$ that is on an internal segment of $B$.
Then $G[Q_1\cup Q_2\cup S\cup P]$ induces a wheel, a contradiction.

\vspace{2ex}

\noindent {\bf Case 6:} For some distinct segments $S$ and $S'$ of
$B$, $N_B(u)\subseteq S$ and $N_B(v)\subseteq S'$.
\\
\\
Let $K_1$ and $K_2$ (resp.\ $K_3$ and $K_4$) be the end cliques of $S$
(resp.\ $S'$). We divide this case in several subcases.

\vspace{2ex}
\noindent{\bf Case 6.1:} $K\not\in\{K_1,K_2,K_3,K_4\}$.
\vspace{2ex}

\noindent We may assume that $K_3\not\in\{K_1,K_2\}$. Let $P_1$, $P_2$ and $P_3$ be the 3 paths obtained by applying Lemma \ref{l:clawCliqueHole} to $K_1$, $K_2$ and $K_3$. Suppose that $N_B(u)$ is a single vertex $u'$. Since (2') does not hold, $v$ has a neighbor in $S'\setminus (K_1\cup K_2)$. But then $G[P_1\cup P_2\cup P_3\cup P\cup S\cup (S'\setminus(K_1\cup K_2))]$ contains a theta. So, by Lemma \ref{node-attach}, $N_B(u)$ is a clique of size 2 in $S$, and similarly $N_B(v)$ is a clique of size 2 in $S'$. By (1) of Lemma \ref{l:extendB}, this contradicts the maximality of $B$.

\vspace{2ex}
\noindent{\bf Case 6.2:} $K_4=K$ and $K\not\in\{K_1,K_2\}$.
\vspace{2ex}

\noindent
{\bf Case 6.2.1:} $k=1$.
\vspace{2ex}

\noindent
By Lemma \ref{l:pyramidSv}, $B$ contains  a pyramid $\Pi =3PC(x_1x_2x_3,v_1)$ such that $S$ and $S'$ are contained in different paths of $\Pi$.
By part (1) of Lemma \ref{l:extendB}, $P$ cannot be a loose crossing of $\Pi$. So by Lemma \ref{crossings},
$P$ is a crosspath of $\Pi$. But this contradicts our choice of $B$ since $k=1$.

\vspace{2ex}

\noindent
{\bf Case 6.2.2:} $k\geq 2$.
\vspace{2ex}

\noindent
For $i\in  \{ 1,2\}$, let $x_i$ be the endnode of $S$ that is in $K_i$, and let $v_i$ and $v_{S'}$ be the endnodes of $S'$.
First suppose that $K_2=K_3$.
Let $P_1=w\ldots w_1, P_2=w\ldots w_2$ and $P_{S'}=w\ldots v_{S'}$ be the three paths obtained by applying Lemma \ref{l:clawSC}
to $S'$ and $K_1$ (where for $i\in \{ 1,2\}$, $P_i\cap K_i=\{ w_i\}$).
Then $G[P_1\cup P_2\cup P_{S'}\cup S]$ is a pyramid $\Pi =3PC(x_2w_2v_{S'},w)$, and $S$ and $S'$ belong to different paths of $\Pi$.
Suppose $v_i=w$ and $N_B(v)=v_i$. If $u$ has a unique neighbor $u'$ in $S$, then $G[\Pi \cup P]$ contains a $3PC(u',w)$,
and otherwise by part (3) of Lemma \ref{l:extendB} our choice of $B$ is contradicted. So either $v_i\neq w$ or $N_B (v)\neq \{ v_i\}$.
But then by Lemma \ref{crossings}, $P$ is a crosspath or a loose crossing of $\Pi$, and therefore by Lemma \ref{l:extendB}
our choice of $B$ is contradicted.

So by symmetry, $K_3\not\in \{ K_1,K_2\}$.
Let $P_1=w\ldots w_1, P_2=w\ldots w_2$ and $P_3=w\ldots w_3$ be the three paths obtained by applying Lemma \ref{l:clawCliqueHole}
to $K_1$, $K_2$ and $K_3$ (so $w\in K$ and  for $i\in \{ 1,2,3\}$, $P_i\cap K_i=\{ w_i\}$).
Let $Q$ be a direct connection from $K_3$ to $P_1\cup P_2$ in $B\setminus K$ and $H$ a hole in $G[P_1\cup P_2\cup P_3\cup S\cup S'\cup Q]$ that contains $S$ and $S'$. Suppose $N_B(v)=\{ w\}$. If $N_B(u)=\{ u'\}$, then $G[H\cup P]$ contains a $3PC(u',w)$. So $N_B(u)$ is a clique of size 2 in $S$, and hence by Lemma \ref{l:extendB} our choice of $B$ is contradicted.
So $N_B(v)\neq \{ w\}$. Now, let us assume that $v_i\neq w$ and that one of the paths $P_1$ and $P_2$ contains a vertex from $K\setminus\{w,v_i\}$. Note that then $P_3\cap S'=\emptyset$. Let $\Pi'$ be a pyramid contained in $G[P_1\cup P_2\cup P_3\cup S\cup Q]$ (this pyramid contains $S$ and its claw has center $w$). Then $G[P\cup P_3\cup Q\cup S']$ contains a crossing of $\Pi'$ with an endnode in $u$, and hence $u$ has two neighbors in $S$ (since $u$ is not adjacent to $w$). If $N_B(v)=\{v'\}\neq\{v_i\}$, then $G[P_1\cup P_2\cup P_3\cup S\cup S'\cup P]$ contains a $3PC(u',w)$, and if $N_B(v)=\{v_i\}$, then our choice of $B$ is contradicted by Lemma \ref{l:extendB}. So, we may assume that $N_B(v)=\{v',v_i\}$,  since otherwise  our choice of $B$ is contradicted by Lemma \ref{l:extendB}. But then $G[P_1\cup P_2\cup S\cup S'\cup P]$ contains a wheel with center $v_i$, a contradiction.

So $v_i=w$ or  $P_1\cap K,P_2\cap K\in\{\{w\},\{w,v_i\}\}$. Then $G[P_1\cup P_2\cup P_3\cup S\cup S'\cup Q]$ contains a pyramid $\Pi$ (whose claw has center $w$ or $v_i$), such that $S$ and $S'$ belong to different paths of $\Pi$. By our choice of $B$ and Lemma \ref{l:extendB}, $P$ cannot be a loose crossing of $\Pi$.
So, by Lemma \ref{crossings}, $P$ is a crosspath of $\Pi$. If the center of the claw of $\Pi$ is $v_i$ and $v_i\neq w$, then $G[P_1\cup P_2\cup P_3\cup S\cup S'\cup P]$ contains a theta or a wheel, a contradiction. So, the center of the claw of $\Pi$ is $w$. Also $w=v_i$, since otherwise our choice of $B$ is contradicted by Lemma \ref{l:extendB}.
This implies that $S'$ is a claw segment of $B$. Let $Q_1$ and $Q_2$ be the paths obtained when Lemma \ref{l:conn3segm+1vertex} is applied to $K_1$ and $S'$ (we assume that $Q_1\cap K_1\neq\emptyset$). Furthermore, if $Q_1$ does not contain $S$, then we can extend $Q_1$ such that it contains  one neighbor of $u$ and such that we do not introduce edges between this new path and $Q_2$. But then, $G[Q_1\cup Q_2\cup S'\cup P]$ contains a wheel or a theta, a contradiction.

\vspace{2ex}
\noindent{\bf Case 6.3:} $K_2=K_4=K$ and $K_1\neq K_3$.
\vspace{2ex}

\noindent
Let $v_i$ (resp. $v_j$) be the endnode of $S$ (resp. $S'$) in $K$, and let $x_S$ (resp. $x_{S'}$) be the other  endnode of $S$ (resp. $S'$).

\vspace{2ex}
\noindent{\bf Case 6.3.1:} $v_i=v_j$.
\vspace{2ex}

\noindent
First, let $k=1$. By Lemma \ref{l:pyramidSS} $B$ contains a pyramid $\Pi =3PC(x_1x_2x_3,v_i)$ such that $S$ and $S'$ are contained in different paths of $\Pi$.
Since $P$ does not satisfy (1) and does not satisfy (2') w.r.t.\ $S$ nor w.r.t.\ $S'$, $P$ is a crossing of $\Pi$.
By the choice of $B$ and since $k=1$, $P$ cannot be a crosspath of $\Pi$. So by Lemma \ref{crossings}, $P$ is a loose crossing of $\Pi$.
But then by part (1) of Lemma \ref{l:extendB}, our choice of $B$ is contradicted.

So, let $k\geq 2$. Let $P_1$ and $P_2$ be paths obtained when Lemma \ref{l:conn4segm} is applied to $S$ and $S'$. Since $P$ does not satisfy (2'),
node  $u$ (resp.\ $v$) has a neighbor in $S\setminus\{v_i\}$ (resp.\ $S'\setminus\{v_i\}$). If $u$ or $v$ is adjacent to $v_i$, then
$G[S\cup S'\cup P\cup P_1\cup P_2]$ contains a wheel with center $v_i$.
Therefore, neither $u$ nor $v$ is adjacent to $v_i$.
Suppose that $u$ has the unique neighbor $u'$ in $S$. If $u'v_i$ is not an edge, then $G[S\cup S'\cup P\cup P_1]$ contains a
$3PC(u',v_i)$. If $u'v_i$ is an edge, then $G[S\cup S'\cup P\cup P_1\cup P_2]$ contains a wheel with center $v_i$  or a theta.
So by Lemma \ref{node-attach}, $N_B(u)$ is a clique of size 2 that belongs to $S\setminus v_i$, and by symmetry
$N_B(v)$ is a clique of size 2 that belongs to $S' \setminus v_i$.
By (1) of Lemma \ref{l:extendB}, this contradicts the maximality of $B$.

\vspace{2ex}
\noindent{\bf Case 6.3.2:} $v_i\neq v_j$.
\vspace{2ex}

\noindent
In particular, $k\geq 2$.
First suppose that $S$ and $S'$ are both clique segments of $B$.
Let $P_1=w\ldots w_1$, $P_2=w\ldots w_2$ and $P_{S'}=w\ldots w_{S'}$ be the three paths obtained by applying Lemma \ref{l:clawSC}
to $S'$ and $K_1$. So $w\not\in \{v_i,v_j\}$. Since (2') does not hold $u$ (resp.\ $v$) has a neighbor in $S\setminus \{v_i\}$
(resp.\ $S'\setminus \{v_j\}$). Let $u'$ be a neighbor of $u$ in $S\setminus \{v_i\}$.
If either $N_B (u)=\{ u'\}$ or $u$ is adjacent to $v_i$, then $G[P_1\cup P_2\cup S\cup (S'\setminus \{ v_j\} )]$ contains a wheel with center $v_i$
or a $3PC(u',w)$. So by Lemma \ref{node-attach}, $N_B(u)$ is a clique of size 2 in $S\setminus K$, and by symmetry $N_{B}(v)$ is a clique of size 2 in $S'\setminus K$.
But then by (1) of Lemma \ref{l:extendB} our choice of $B$ is contradicted.

So w.l.o.g.\ we may assume that $S$ is a claw segment.
Let $Q$ be a direct connection from $K_1$ to $K_3$ in $B\setminus K$. Let $S_1$ be a segment of $B$ distinct from $S$ that has endnode
$v_i$. Let $Q_1$ be a direct connection from $S_1$ to $Q$ in $B\setminus K$. Then $G[S\cup S'\cup S_1\cup Q\cup Q_1]$
is a pyramid $\Pi =3PC(x_1x_2x_3,v_i)$, in which $S$ and $S'$ are contained
in different paths of $\Pi$.
$P$ cannot be a loose crossing of $\Pi$, since otherwise by (1) of Lemma \ref{l:extendB}
our choice of $B$ is contradicted. Therefore by Lemma \ref{crossings}, $P$ is a $v_i'$-crosspath of $\Pi$, where $v_i'$ is the neighbor
of $v_i$ in $S$ (since $v$ has a neighbor in $S'\setminus \{v_j\}$). In particular, $uv_i'$ is an edge and $N_B(u)\subseteq \{ v_i',v_i\}$.
By \ref{i:11} there exists a leaf segment $S_2$ of $B$ with endnode $v_k$ such that either $v_k\in K\setminus \{ v_i,v_j\}$,
or $v_k=v_j$ and $S_2\neq S'$. Let $Q_2$ be a direct connection from $S_2$ to $\Pi$. Then $G[\Pi \cup S_2\cup Q_2]$
contains a $3PC(v_i',v_j)$ (if $N_B(u)=\{v_i'\}$ and $j\neq k$) or a wheel with center $v_i$ (otherwise).

\vspace{2ex}
\noindent{\bf Case 6.4:} $K_2=K_4=K$ and $K_1=K_3$.
\vspace{2ex}

\noindent
Let $v_i$ (resp. $v_j$) be the endnode of $S$ (resp. $S'$) in $K$, and let $x_S$ (resp. $x_{S'}$) be the other  endnode of $S$ (resp. $S'$).

\vspace{2ex}
\noindent{\bf Case 6.4.1: }$k=1$.
\vspace{2ex}

\noindent
Then by \ref{i:9} $R$ has no branches. By (i) $B$ contains a pyramid $\Pi =3PC(x_Sx_{S'}x,v_1)$ where $S$ and $S'$ are paths of $\Pi$.
Since $P$ does not satisfy (1) and does not satisfy (2') w.r.t.\ $S$ nor w.r.t.\ $S'$, $P$ is a crossing of $\Pi$.
By the choice of $B$ and since $k=1$, $P$ cannot be a crosspath of $\Pi$. So by Lemma \ref{crossings}, $P$ is a loose crossing of $\Pi$.
But then by part (1) of Lemma \ref{l:extendB}, our choice of $B$ is contradicted.

\vspace{2ex}
\noindent{\bf Case 6.4.2:} $k\geq 2$.
\vspace{2ex}

\noindent
Then by \ref{i:9}, $v_i\neq v_j$ and w.l.o.g.\ $v_i$ is an endnode of a leaf segment $S_1\neq S$.
Let $Q$ be a direct connection from $S_1$ to $K_1$ in $B\setminus K$. Then $S,S',S_1$ and $Q$ induce a pyramid $\Pi =3PC(x_Sx_{S'}x,v_i)$
(where $x$ is an endnode of $Q$) such that $S$ and $S'v_i$ are paths of $\Pi$.
Since $P$ does not satisfy (1) and it does not satisfy (2') w.r.t.\ $S$ nor w.r.t.\ $S'$, $P$ is a crossing of $\Pi$.
By (1) of Lemma \ref{l:extendB} and our choice of $B$, $P$ cannot be a loose crossing of $\Pi$.
So by Lemma \ref{crossings}, $P$ is a crosspath of $\Pi$.
If $P$ is a $v_j$-crosspath of $\Pi$ then by part (3) of Lemma \ref{l:extendB}, our choice of $B$ is contradicted.
So for the neighbor $v_i'$ of $v_i$ in $S$,
$P$ is a $v_i'$-crosspath of $\Pi$. In particular, $u$ is adjacent to $v_i'$ and $N_{B} (u)\subseteq \{ v_i,v_i'\}$, and $N_{B} (v)$ is a clique of size 2 of
$S'\setminus \{v_j\}$.
By \ref{i:11} there exists a leaf segment $S_2$ with endnode $v_k\in K\setminus \{ v_i\}$ such that $S_2\neq S'$.
Let $Q_2$ be a direct connection from $S_2$ to $\Pi \setminus (S \cup S' \cup K)$.
But then $G[(\Pi \setminus \{ x_{S'} \} )\cup S_2\cup Q_2]$ either contains a $3PC(v_i',v_j)$ (if $k=j$ and $uv_i$ is not an edge)
or a wheel with center $v_i$ (otherwise).
\end{proof}

Let $\Pi=3PC(x_1x_2x_3,y)$ be a pyramid contained in a graph $G$.
A {\em hat} of $\Pi$ is a chordless path $P=p_1\ldots p_k$ in $G \setminus \Pi$
such that $p_1$ and $p_k$ both have a single neighbor in $\Pi$ and they are adjacent to
different nodes of $\{ x_1,x_2,x_3\}$, and no interior node of $P$ has a neighbor in $\Pi$.

\begin{lemma}\label{hat}
Let $G$ be a (theta, wheel)-free graph. If $G$ contains a pyramid with a hat, then $G$ has a clique cutset.
\end{lemma}

\begin{proof}
Let $P=p_1\ldots p_k$ be a hat of $\Pi=3PC(x_1x_2x_3,y)$ contained in $G$, with w.l.o.g.\ $N_{\Pi}(p_1)=\{ x_1\}$ and
$N_{\Pi}(p_k)=\{ x_2\}$. Assume that $G$ does not have a clique cutset. Then by Lemma \ref{diamond}, $G$ is diamond-free.
Let $S$ be the set comprised of $\{ x_1,x_2,x_3\}$ and all nodes $u\in G\setminus \Pi$ such that $N_{\Pi}(u)=\{ x_1,x_2,x_3\}$.
Since $G$ is diamond-free, $S$ is a clique. Let $Q=q_1\ldots q_l$ be a direct connection from $P$ to
$\Pi \setminus \{ x_1,x_2,x_3\}$ in $G \setminus S$. We may assume w.l.o.g.\ that a hat $P$ and direct connection $Q$ are chosen so that $|V(P)\cup V(Q)|$
is minimized.

By Lemma \ref{node-attach}, $q_l$ either has a single neighbor in $\Pi$ or $N_{\Pi} (q_l)$ are two adjacent nodes of a path of $\Pi$.
If a node $q_i$, $i<l$, is adjacent to a node of $\{ x_1,x_2,x_3\}$, then by definition of $Q$, $q_i$ has a single neighbor in $\Pi$.
If at least two nodes of $\{ x_1,x_2,x_3\}$ have a neighbor in $Q\setminus q_l$, then a subpath of $Q\setminus q_l$ is a hat of
$\Pi$, contradicting the minimality of $P\cup Q$. So at most one node of $\{ x_1,x_2,x_3\}$ has a neighbor in $Q\setminus q_l$.
Suppose $x_i$, for some $i\in \{ 1,2,3\}$, has a neighbor in $Q\setminus q_l$, and let $q_t$ be such a neighbor with highest index.
Then $N_{\Pi} (q_l)\subseteq S_i$, since otherwise $q_t\ldots q_l$ is a crossing of $\Pi$ that contradicts Lemma \ref{crossings}.
If $i=3$ then a subpath of $(P\setminus p_k) \cup Q$ or $(P\setminus p_1) \cup Q$ is a hat of $\Pi$, contradicting the minimality of $P\cup Q$.
So w.l.o.g.\ $i=1$. But then $G[(\Pi \setminus y_2)\cup P\cup Q]$ contains a wheel with center $x_1$.
Therefore, no node of $\{ x_1,x_2,x_3\}$ has a neighbor in $Q\setminus q_l$.

Without loss of generality we may assume that $P\cup \{ q_1\}$ contains a chordless path $P'$ from $p_1$ to $q_1$ that does not contain $p_k$.
Then $N_{\Pi} (q_l)\subseteq S_1$, since otherwise the path induced by $P'\cup Q$ is a crossing of $\Pi$ that contradicts Lemma \ref{crossings}.
If $N_{\Pi} (q_l)=\{ y\}$ then $P'\cup Q\cup S_1\cup S_3$ induces a $3PC(x_1,y)$. So $q_l$ has a neighbor in $S_1\setminus \{ x_1,y\}$.
If $p_1$ is the unique neighbor of $q_1$ in $P$, then $G[P\cup Q\cup (\Pi \setminus y_2)]$ contains a wheel with center $x_1$.
So $P\cup \{ q_1\}$ must contain a chordless path $P''$ from $p_k$ to $q_1$ that does not contain $p_1$. But then the path induced by
$P''\cup Q$ is a crossing of $\Pi$ that contradicts Lemma \ref{crossings}.
\end{proof}

\section{Proof of Theorem \ref{main}}
\label{sec:proofM}

A {\em strip} is a triple $(H, A, A')$ that satisfies the
following:
\begin{itemize}
\item[(i)] $H$ is a graph and $A$ and $A'$ are disjoint non-empty
  cliques of $H$;
\item[(ii)] every vertex of $H$ is contained in a chordless path of $H$
  whose one endnode is in $A$, the other is in $A'$, and no
  interior node is in $A\cup A'$ (such a path is called an {\em
    $AA'$-rung}).
\end{itemize}

Let $B$ be a P-graph with special clique $K$, and let $V_0$ be the set of all vertices of $B$ that are the unique vertex of some segment of length zero. A \emph{strip system} $\mathcal S$ is any graph obtained from $B$ as follows:

\begin{itemize}
\item for every segment $S = u \dots v$ of $B$ of length at least 1, let $(H_S, Q_{u, S}, Q_{v, S})$ be
  a strip, such that $Q_{u,S}\cap S=\{u\}$ and $Q_{v,S}\cap S=\{v\}$;

\item $V(\mathcal S)$ is the union of vertices of $H_S$, for all segments $S$ of $B$ of length at least 1, and $V_0$;

\item if $S=u\ldots v$, $u\in K$, is a claw segment of $B$, then $Q_{u,S}=\{u\}$;

\item for segments $S$ and $S'$ of length at least 1, if $S\cap S'=\emptyset$, then $V(H_S)\cap V(H_{S'})=\emptyset$;

\item  a clique $Q_{x, S}$ is complete to a clique $Q_{x', S'}$
  whenever $x$ and $x'$ are in the same clique of $\cal K$;

\item  a clique $Q_{x, S}$ is complete to $x'$
  whenever $x$ and $x'$ are in the same clique of $\cal K$ and $x'\in V_0$;

\item these are the only edges of the strip system.
\end{itemize}

Furthermore, for a clique $K_1\in\mathcal K$, we denote $Q_{K_1}=\bigcup_{u\in K_1}Q_{u,S}\cup K_1$
(where $S$ is a segment of length at least 1 that contains $u$).

Note that any P-graph can be seen as a strip system, where every segment of length at least 1
is replaced by a strip equal to the segment.  So, strip system can be
seen as a way to thicken a P-graph.  In the other direction, consider
a graph $T$ induced by $V_0$ and vertices of one rung from every strip of a
strip system ${\cal S}$. We say that $T$ is a {\em template} of
${\cal S}$.  Note that in particular $B$ is a strip system with unique
template, namely $B$.


\begin{lemma}\label{template}
Let $G$ be a (theta, wheel)-free graph. Then
every template of a strip system of $G$ is a P-graph.
\end{lemma}

\begin{proof}
We claim that given a P-graph $B$ and a strip system obtained from $B$ (that is contained in $G$), replacing one segment $S=u\ldots v$ of $B$ by a corresponding rung $S'=u'\ldots v'$ yields another P-graph $B'$. The lemma then follows from this claim by induction on the number of segments. So let us prove the claim.

Let $K$ be the special clique of $B$ and $R$ its skeleton. If $u$ or $v$, say $u$, is in $K$, then let $K'=\{u'\}\cup K\setminus\{u\}$; otherwise let $K'=K$.

By \cite{harary.holzmann:lgbip} a graph is (claw,diamond)-free if and only if it is the line graph of a triangle-free graph. So, $B\setminus K$ is (claw,diamond)-free, and hence the same holds for $B'\setminus K'$, i.e.\ $B'\setminus K'$ is the line graph of a triangle-free graph $R'$. Observe that $R'$ can be obtained from $R$ by changing the length of a single branch or limb. Furthermore, in this way no branch of length 1 is obtained since the two cliques of any strip are disjoint. Therefore, $R'$ satisfies all conditions of the definition of a skeleton, except possibly the ones that are concerned with the lengths of the limbs. So, we only need to check that $R'$ satisfies \ref{i:8}, which is true by Lemma \ref{l2}.
\end{proof}

We are ready to prove Theorem \ref{main}.

\vspace{2ex}

\noindent
{\em Proof of Theorem \ref{main}:}
Let $G$ be a (theta, wheel)-free graph, and assume that $G$ does not
have a clique cutset and that it is not a line graph of triangle-free
chordless graph.  By Lemma \ref{diamond}, $G$ is diamond-free and by
Theorem \ref{dt1-p1}, $G$ contains a pyramid, and hence a long pyramid (since $G$ is wheel-free).
So, by Lemma \ref{pyrBas},
$G$ contains a P-graph. Let $B$ be a P-graph contained in $G$ with
maximum size of the special clique $K$, say $|K|=k$, and such that out
of all P-graphs with special clique of size $k$ it has the maximum
number of segments. Let $\mathcal K$ be the set that includes all big
cliques of $B$ and $K$, and let $R$ be the skeleton of $B$.
Furthermore, let $\mathcal S$ be a maximal (w.r.t.\ inclusion) strip
system obtained from $B$.

\vspace{2ex}

\noindent{\bf Claim 1.} For every $w\in G\setminus\mathcal S$ either
for some  clique $K_1\in \mathcal K$,  $N_{\mathcal S}(w)=Q_{K_1}$,
or for some segment
$S$ of $B$ of length at least 1, $N_{\mathcal S}(w)\subseteq H_S$.

\vspace{2ex}

\noindent{\it Proof of Claim 1.}
Suppose not. Observe that if for some $K_1\in \mathcal K$, $w$ has two distinct neighbors in $Q_{K_1}$, then since $G$ is
diamond-free, $w$ is complete to $Q_{K_1}$.

First suppose that $w$ is adjacent to a vertex $v\in V_0$. By Lemma \ref{star}, $X=(N_G(v)\setminus (\{ w \} \cup \mathcal S)) \cup \{ v\}$
is not a star cutset of $G$, so there exists a chordless path $P=w\ldots w'$ in $G\setminus (\mathcal S \cup X)$
such that $w'$ has a neighbor $u'$ in $\mathcal S \setminus \{ v \}$ and no interior node of $P$ has a neighbor in $\mathcal S$.
By definition of a strip  and $\mathcal S$, there is a template of $\mathcal S$ that contains $u'$ and $v$.
By Lemma \ref{template} we may assume w.l.o.g.\ that $B$ contains $u'$ and $v$.
By Lemma \ref{path-attach} applied to $P$ and $B$, and since $w'$ is not adjacent to $v$, $N_B(P)\subseteq K'\in \mathcal K$.
In particular, $K'\in \mathcal K \setminus \{ K\}$, $u',v\in K'$ and $N_B(w')=\{ u'\}$.
By Lemma \ref{newsec4l}, $B$ contains a pyramid $\Pi=3PC(u'vx,y)$, with $x\in K'$ and $y\in K$.
If $N_{\Pi} (w)=\{ v\}$ then $P$ is a hat of $\Pi$, contradicting Lemma \ref{hat}.
So there exists $v'\in N_{\Pi} (w)\setminus \{ v\}$. By Lemma \ref{node-attach}, $N_{\Pi} (w)$ is a maximal clique of $\Pi$.
If $N_{\Pi} (w) \neq \{ u' ,v,x\}$ then $G[\Pi \cup P]$ contains a wheel with center $v$. So $N_{\Pi} (w) = \{ u' ,v,x\}$, and hence $w$ is complete
to $Q_{K'}$. It follows that $w$ has a neighbor $u''$ in $\mathcal S \setminus Q_{K'}$. Let $B'$ be a template of $\mathcal S$
that contains $v$ and $u''$. By Lemma \ref{template}, $B'$ is a P-graph. By Lemma \ref{node-attach}, $N_{B'} (w)$ is a maximal clique of $B'$,
and in particular $vu''$ is an edge. It follows that for some $K''\in \mathcal K \setminus \{ K',K\}$, $w$ is complete to $Q_{K''}$.
By Lemma \ref{l:clawHole} applied to $B$ and $v$, there exists a hole $H$ in $B$ that contains $v$, and hence it contains a vertex of $K'\setminus \{ v\}$
and a vertex of $K''\setminus \{ v\}$. But then $(H,w)$ is a wheel, a contradiction.

Therefore, $w$ is not adjacent to a vertex of $V_0$. It follows that there exist distinct segments $S$ and $S'$ of $B$, both of length at least 1,
such that $w$ has a neighbor $u$ in $H_S$,  a neighbor $v\in H_{S'}$, and there is no clique $K_1\in \mathcal K$ such that $u$ and $v$ are both in $Q_{K_1}$.
Let $B'$ be a template of $\mathcal S$ that contains $u$ and $v$ (it exists by definition of a strip and $\mathcal S$).
But then by Lemma \ref{template}, $B'$ and $w$ contradict Lemma \ref{node-attach}.
This completes the proof of Claim 1.

\vspace{2ex}

\noindent{\bf Claim 2.} Let $S$ be a segment of $B$ of length at least 1 with endnodes $u\in K_1$ and  $v\in K_2$, where $K_1$ and $K_2$ are distinct cliques of
$\mathcal K$, $K_1\neq K$, and let $(H_S,Q_{u,S},Q_{v,S})$ be the corresponding strip of $\mathcal S$.
Then $G\setminus \mathcal S$ cannot contain a chordless path $P=w_1\ldots w_2$ such that the following hold:
\begin{itemize}
\item $N_{\mathcal S} (w_1)=Q_{K_1}$,
\item $N_{\mathcal S} (w_2)=Q_{K_2}$, or $N_{\mathcal S} (w_2) \subseteq H_S$ and $w_2$ has a neighbor in $H_S\setminus Q_{K_1}$, and
\item no interior node of $P$ has a neighbor in $\mathcal S \setminus Q_{u,S}$.
\end{itemize}

\vspace{2ex}

\noindent{\it Proof of Claim 2.}
Assume such a path exists. Let $H_{S}'=H_S\cup P$ and $Q'_{u,S}=Q_{u,S}\cup \{ w_1\}$.
If $N_{\mathcal S} (w_2)=Q_{K_2}$ and either $K_2\neq K$ or $k>1$, then let $Q'_{v,S}=Q_{v,S}\cup \{ w_2\}$,
and otherwise let $Q'_{v,S}=Q_{v,S}$.
Since $w_2$ has a neighbor in $H_S\setminus Q_{K_1}$, $H'_{S}$ contains a rung with endnode $w_1$ that contains $P$,
so $(H'_{S},Q'_{u,S},Q'_{v,S})$ is a strip. Since, by maximality of $\mathcal S$, $\mathcal S'=\mathcal S \cup P$ cannot be a strip system,
it follows that $S$ is a claw segment (so $K_2=K$) and $N_{\mathcal S} (w_2)=K$ and $k>1$.
Since $S$ is a claw segment of $B$, $Q_{v,S}=\{ v\}$, and there exists another leaf segment $S'$ of $B$ with endnode $v$.
Suppose that a node $u_1$ of $Q_{u,S}$ has a neighbor in interior of $P$. Let $S_1$ be a rung of $H_S$ that contains $u_1$.
By Lemma \ref{template}, $B'=(B\setminus S)\cup S_1$ is a $P$-graph where $S_1$ is a claw segment, so by (viii) of the definition of skeleton, $u_1v$ is not an edge.
Let $H'$ be a hole of $B'$ that contains $S_1$ and $S'$. But then $G[H'\cup (P\setminus w_1)]$ contains a $3PC(u_1,v)$, a contradiction.
Therefore, no node of $\mathcal S$ has a neighbor in interior of $P$. But then by (2) of Lemma \ref{l:extendB}, the choice of $B$ is contradicted.
This completes the proof of Claim 2.

\vspace{2ex}

\noindent{\bf Claim 3.}
For a clique $K_1\in \mathcal K\setminus \{ K\}$, there cannot exist a vertex $w$ of $G\setminus \mathcal S$ such that
$N_{\mathcal S} (w)=Q_{K_1}$.

\vspace{2ex}

\noindent{\it Proof of Claim 3.}
Suppose such a vertex exists. Let $K'$ be a maximal clique of $G\setminus \{w\}$ that contains $Q_{K_1}$.
Note that since $G$ is diamond-free, no node of $G\setminus (K'\cup \{ w\})$ is complete to $Q_{K_1}$.
Since $K'$ cannot be a clique cutset of $G$, there exists a chordless path $P=w\ldots w'$ in $G\setminus (\mathcal S \cup K')$
such that $w'$ has a neighbor $u'$ in $\mathcal S \setminus Q_{K_1}$, no node of $P\setminus\{w\}$ is complete to $Q_{K_1}$,
and no interior node of $P$ has a neighbor in $\mathcal S \setminus Q_{K_1}$. By Claim 1 one of the following two cases hold.

\vspace{2ex}

\noindent{\bf Case 1:} For some segment $S$ of $B$ of length at least 1, $N_{\mathcal S} (w')\subseteq H_S$.
\\
\\
First suppose that $S$ has an endnode $u\in K_1$ and an endnode $v\in K_2$, for $K_2\in \mathcal K\setminus \{ K_1\}$.
By Claim 2, a node of $Q_{K_1}\setminus Q_{u,S}$ must have a neighbor in $P\setminus w$. Let $w''$ be a node of $P\setminus \{w\}$
closest to $w'$ that has a neighbor in $Q_{K_1}\setminus Q_{u,S}$. So, since $G$ is diamond-free and $|K_1|\geq 3$,
$N_{\mathcal S} (w'')=\{ u''\}$, where $u'' \in Q_{K_1}\setminus Q_{u,S}$. Let $B'$ be a template of $\mathcal S$ that contains
$u'$ and $u''$. By Lemma \ref{template} $B'$ is a $P$-graph, and so $B'$ and the $w''w'$-subpath of $P$  contradict Lemma \ref{path-attach}.

So $S$ does not have an endnode in $K_1$.
Let $w''$ be a node of $P$ closest to $w'$ that has a neighbor in $Q_{K_1}$. Let $u''$ be a neighbor of $w''$ in $Q_{K_1}$, and
let $B'$ be a a template of $\mathcal S$ that contains $u'$ and $u''$.
By Lemma \ref{template} $B'$ is a $P$-graph, and so $B'$ and the $w''w'$-subpath of $P$  contradict Lemma \ref{path-attach}.

\vspace{2ex}

\noindent{\bf Case 2:} For some clique $K_2\in \mathcal K \setminus \{ K_1\}$, $N_{\mathcal S} (w')= Q_{K_2}$.
\\
\\
First suppose that there exists a segment $S$ of $B$ of length at least 1 with endnode $u\in K_1$ and an endnode $v\in K_2$.
Then by Claim 2, a node of $Q_{K_1} \setminus Q_{u,S}$ has a neighbor in $P$. Let $w''$ be the node of $P$ closest to
$w'$ that has a neighbor in $Q_{K_1}\setminus Q_{u,S}$. Then $N_{\mathcal S} (w'')=\{ u''\}$.
Let $B'$ be a template of $\mathcal S$ that contains $S$ and $u''$. By Lemma \ref{template} $B'$ is a $P$-graph, and so
$B'$ and the $w''w'$-subpath of $P$ contradict Lemma \ref{path-attach}.

So no segment of $B$ of length at least 1 has an endnode in $K_1$ and an endnode in $K_2$.
Let $w''$ be the node of $P$ closest to $w'$ that has a neighbor $u_1\in Q_{K_1}\setminus Q_{K_2}$. Let $B'$ be a template of $\mathcal S$ that
contains $u_1$. Then by Lemma \ref{template}, $B'$ and the $w''w'$-subpath of $P$ contradict Lemma \ref{path-attach}.

This completes the proof of Claim 3.





\vspace{2ex}

\noindent{\bf Claim 4.}
Let $S$ be a clique segment of $B$ with endnode $v\in K$.
Then $G\setminus \mathcal S$ cannot contain a chordless path $P=w_1\ldots w_2$ such that the following hold:
\begin{itemize}
\item $w_1$ has a neighbor in $H_S\setminus Q_K$,
\item $N_{\mathcal S} (w_2)=Q_{K}$, and
\item no interior node of $P$ has a neighbor in $\mathcal S \setminus Q_{v,S}$.
\end{itemize}

\vspace{2ex}

\noindent{\it Proof of Claim 4.}
Assume such a path exists. By Lemma \ref{template}, w.l.o.g.\ we may assume that $w_1$ has a neighbor in $S\setminus K$.
Let $u$ be an endnode of $S$ different from $v$, and let $K_1\in \mathcal K\setminus \{ K\}$
such that $u\in K_1$. By Claim 3, $N_{\mathcal S} (w_1)\neq Q_{K_1}$, and so by Claim 1,  $N_{\mathcal S} (w_1) \subseteq H_S$.
Let $H'_S=H_S \cup P$ and $Q'_{v,S}=Q_{v,S} \cup \{ w_2\}$.
Then $(H'_S,Q'_{v,S},Q_{u,S})$ is a strip and $\mathcal S'=\mathcal S \cup P$ is a strip system that contradicts our choice
of $\mathcal S$.
This completes the proof of Claim 4.

\vspace{2ex}

\noindent{\bf Claim 5.} For every connected component $C$ of $G\setminus \mathcal S$, there exists a segment $S$ of $B$
of length at least 1 such that
$N_{\mathcal S} (C) \subseteq H_S$.

\vspace{2ex}

\noindent{\it Proof of Claim 5.}
Suppose that a connected component  $C$ of $G\setminus \mathcal S$ does not satisfy the stated property.
Since $Q_K$ is not a clique cutset, some node of $C$ has a neighbor in $\mathcal S \setminus Q_K$.
So by Claims 2 and 3 some node $w_1$ of $C$ has a neighbor in $H_S\setminus Q_K$ for some segment $S$ of $B$ of length at least 1.
So there exists a chordless path $P=w_1\ldots w_2$ in $C$ such that $w_2$ has a neighbor in $\mathcal S\setminus H_S$. We choose $P$ to be a minimal such path.

First suppose that $S$ is an interior segment of $B$, and let $u\in K_1$ and $v\in K_2$ be endnodes of $S$, where $K_1,K_2\in \mathcal K \setminus \{ K\}$.
By Lemma \ref{template} w.l.o.g.\ we may assume that $w_1$ has a neighbor in $S$ and $w_2$ has a neighbor in $B\setminus S$.
By the choice of $P$, no interior node of $P$ has a neighbor in $B$. But then by Lemma \ref{path-attach}, $w_2$ is complete
to $K_1$ or $K_2$, say $K_1$. By Claim 1 $N_{\mathcal S} (w_2)=Q_{K_1}$, contradicting Claim 3. Therefore $S$ is a leaf segment of $B$.

Let $u\in K_1$ and $v\in K$ be the endnodes of $S$. By the choice of $P$, no interior node of $P$ has a neighbor in $\mathcal S \setminus Q_{v,S}$.
Suppose $w_2$ has a neighbor in $\mathcal S \setminus Q_K$.
Then by Lemma \ref{template}, w.l.o.g.\ we may assume that $w_1$ has a neighbor in $S\setminus \{ v\}$ and $w_2$ has a neighbor in $B\setminus (K\cup S)$.
By Lemma \ref{path-attach} and Claims 1 and 3, an interior node of $P$ is adjacent to $v$. Let $w_1'$ be the interior node of $P$ closest to $w_2$ that is
adjacent to $v$. By Lemma \ref{path-attach} applied to $w_1'w_2$-subpath of $P$,  for some leaf segment $S'$ of $B$ with endnode $v$,
$w_2$ has a neighbor in $S'\setminus v$.
Let $K_2$ be the clique in $\mathcal K \setminus \{ K\}$ that contains a node of $S'$. Recall that interior nodes of $P$ do not have neighbors in $B\setminus v$.
Also by Claims 1 and 3 and by Lemma \ref{node-attach}, $N_B(w_1)$ (resp. $N_B(w_2)$) is either a single node or an edge of $S$ (resp. $S'$).
Suppose $k\geq 2$. Then by \ref{i:9} of the definition of skeleton, $K_1\neq K_2$. Let $P_1$ and $P_2$ be the paths obtained by applying Lemma \ref{l:conn4segm}
to $S$ and $S'$. Then $G[S\cup S' \cup P_1\cup P_2\cup P]$ contains a theta or a wheel with center $v$.
So $k=1$.

By Lemma \ref{l:pyramidSS} let $\Pi$ be a pyramid contained in $B$ such that $S$ and $S'$ are contained in different
paths of $\Pi$.
If $w_1$ is adjacent to $v$ then $G[\Pi \cup P]$ contains a wheel with center $v$. So $w_1$ is not adjacent to $v$ and by symmetry neither is $w_2$.
If both $w_1$ and $w_2$ have unique neighbors in $\Pi$, then $G[\Pi \cup P]$ contains a wheel with center $v$ or a theta.
So w.l.o.g.\ $N_B(w_1)=\{ w_1',w_1''\}$ where $w_1'w_1''$ is an edge of $S$.
Then $G[\Pi \cup P]$ contains a pyramid $\Pi '=3PC(w_1w_1'w_1'', v)$. But then, by Lemma \ref{crossings},
$P\setminus \Pi'$ is a crosspath of $\Pi '$ contradicting our choice of $B$ (since $k=1$).
Therefore, $N_{\mathcal S} (w_2)\subseteq Q_K$.

Since $w_2$ has a neighbor outside $S$, $k>1$. Let $v_2$ be a neighbor of $w_2$ in $K\setminus \{ v\}$.
Let $w$ be a node of $B\setminus K$ adjacent to $v_2$, and let $Q$ be a direct connection in $B\setminus K$ from $w$ to $K_1$.
By Lemma \ref{template} w.l.o.g.\ $w_1$ has a neighbor in $S\setminus K$. Note that by Claims 1 and 3, $N_B(w_1)\subseteq S$.
By Lemma \ref{node-attach}, $N_B(w_1)$ is a clique of size 1 or 2 in $S$.
If $v$ has a neighbor in $P$, then $G[S\cup P\cup Q]$ contains a theta or a wheel.
So $v$ has no neighbor in $P$. Then by Lemma \ref{path-attach}, $w_2$ is complete to $K$, and hence by Claim 1, $N_{\mathcal S} (w_2)=Q_K$.
By Claim 4, $S$ is a claw segment of $B$. So there is a node $w'$ of $B\setminus (K\cup S)$ adjacent to $v$. Let $Q'$ be a direct connection in $B\setminus K$
from $w'$ to $K_1$, and let $w_1'$ be the neighbor of $w_1$ in $S\setminus K$ which is the closest to $K_1$. If $w'_1$ is adjacent to $v$, then $G[S\cup P\cup Q]$ is a wheel with center $v$. So,
$w_1'$ is not adjacent to $v$, and hence by Lemma \ref{l:extendB}, $w_1'$ is the unique neighbor of $w_1$ in $S$. But then $G[S\cup P\cup Q']$ is a theta.
This completes the proof of Claim 5.

\vspace{2ex}

Suppose $G \neq B$. Then by Claim 5, there exists a segment $S$ of $B$ of length at least 1
such that either $H_S\neq S$ or a node of $G\setminus \mathcal S$ has a neighbor in $H_S$.
Let $\mathcal C$ be the union of all connected components $C$ of $G\setminus \mathcal S$ that have a node with a neighbor in $H_S$.
By Claim 5, $N_{\mathcal S}(\mathcal C)\subseteq H_S$.
If $S$ is not a claw segment of $B$, then
 $(H_S\cup \mathcal C,G\setminus (H_S\cup \mathcal C))$ is a 2-join of $G$. So we may assume that $S$ is a claw segment of $B$ with an endnode $u\in K$.
 Then, by Claim 5, $((H_S\setminus\{u\}) \cup \mathcal C,(G\setminus (H_S)\cup \mathcal C)\cup \{u\})$ is  a 2-join of $G$
 (note that $Q_{u,S}=\{ u\}$ and by \ref{i:8} of the definition of skeleton, every rung of $H_S$ is of length at least 2).\hfill$\Box$

%

\section{Recognition algorithm}
\label{sec:reco}

In this section we give a recognition algorithm and a structure
theorem for the class of (theta,wheel)-free graph. For this, most of
the necessary work is already done in \cite{twf-p1} (see Sections 6
and 7, where all important steps in the proof are given for
(theta,wheel)-graphs).

To obtain a recognition algorithm for (theta,wheel)-free graphs we
modify the algorithm given in Theorem 7.6 of \cite{twf-p1} for
only-pyramid graphs. In fact, the only modification that should be
made is the change of the subroutine that checks whether a graph is
basic. A recognition algorithm for basic (theta,wheel)-free graphs is
given in the following lemma.

\begin{lemma}
  \label{l:testBasic}
  There is an $O(n^2m)$-time algorithm that decides whether an input graph is
  the line graph of a triangle-free chordless graph or a P-graph.
\end{lemma}

\begin{proof}
  By Lemma 7.4 from \cite{twf-p1}, there is an $O(n^2m)$-time
  algorithm that decides whether an input graph is the line graph of a
  triangle-free chordless graph. So, it is enough to give an
  $O(n^2m)$-time algorithm that decides whether an input graph is a
  P-graph.

  First, in time $O(n^2m)$ we can find the set $S$ of all centers of
  claws in~$G$. If $S=\emptyset$, or $S$ does not induce a clique,
  then $G$ is not a P-graph. So, assume that $S$ induces a non-empty
  clique. Next, let $K$ be a maximal clique of $G$ that contains $S$, unless $|S|=1$, in which case take $K=S$ if the vertex of $S$ is not contained in a clique of size 3, or $K$ is a maximal clique of size at least 3 that contains $S$ otherwise.
  Now, let $G'$ be the graph
  obtained from $G$ by removing all vertices of $K$ (and edges
  incident with them). Using Lemma 7.4 from \cite{twf-p1} we decide
  (in time $O(n^2m)$) whether $G'$ is a line graph of a triangle-free
  chordless graph, and if it is find $R$ such that $G'=L(R)$ (if $G'$
  is not a line graph of a triangle-free chordless graph, then $G$ is
  not a P-graph). Now, we check whether $R$ is a $k$-skeleton, where
  $k=|K|$. To do this, first we find all pendant edges of $R$.  We
 name vertices of $K$ with numbers 1 to $k$, and give labels to the
 pendant edges of $R$ according to their neighbor in $K$.  We test
  whether (i), (vi), (vii), (viii) and (xi) in the definition of a
  $k$-skeleton are satisfied. Next, we check if (iii) is satisfied (in
  time $O(n(n+m))$) and then if (iv) is satisfied (in time
  $O(n^2(n+m))$). To check (v), note that an edge $e$ is contained in
  a cycle of $R$ if and only if $R\setminus e$ is connected, that is
  (v) can be check in time $O(m(n+m))=O(n^2m)$. Branches and limbs of
  $R$ can be found in time $O(n+m)$ and the number of them is
  $O(m)$. Hence, (ii) and (ix) can be checked in time
  $O(n+m+m^2)=O(n^2m)$.  Finally, for an attaching vertex $x$ of $R$
  all $x$-petals can be found in time $O(n+m)$, and hence (x) can be
  be checked in time $O(n(n+m))$.
\end{proof}

By the previous lemma, recognition of basic (theta,wheel)-free graphs can be done in the same running time as the recognition of basic only-pyramid graphs (used in \cite{twf-p1}). Hence, the recognition algorithm for (theta,wheel)-free graphs, that was explained above, has the same running time as the algorithm given in Theorem 7.6 of \cite{twf-p1}. This proves Theorem \ref{th:reco}.

As in \cite{twf-p1}, our decomposition theorem for (theta,wheel)-free graphs can be turned into
a structure theorem as follows.

Let $G_1$ be a graph that contains a clique $K$ and $G_2$ a graph
that contains the same clique $K$, and is node disjoint from $G_1$
apart from the nodes of $K$.  The graph $G_1 \cup G_2$ is the graph
obtained from $G_1$ and $G_2$ by \emph{gluing along a clique}.

Let $G_1$ be a graph that contains a path  $a_2 c_2 b_2$ such that
$c_2$ has degree 2, and such that $(V(G_1)\sm \{a_2, c_2, b_2\},
\{a_2, c_2, b_2\})$ is a consistent  almost 2-join of $G_1$ (consistent almost 2-join
is a special type of almost 2-joins -- for the definition see \cite{twf-p1}). Let
$G_2, a_1, c_1, b_1$ be defined similarly.  Let $G$ be the graph built
on $(V(G_1)\sm \{a_2, c_2, b_2\})\cup (V(G_2)\sm \{a_1, c_1, b_1\})$
by keeping all edges inherited from $G_1$ and $G_2$, and by adding all
edges between $N_{G_1}(a_2)$ and $N_{G_2}(a_1)$, and all
edges between $N_{G_1}(b_2)$ and $N_{G_2}(b_1)$.  Graph $G$ is said to be
\emph{obtained from $G_1$ and $G_2$ by consistent 2-join
  composition}.   Observe that $(V(G_1)\sm \{a_2, c_2, b_2\},
V(G_2)\sm \{a_1, c_1, b_1\})$ is a 2-join of $G$ and that $G_1$ and
$G_1$ are the blocks of decomposition of $G$ with respect to this
2-join.

Using the results from \cite{twf-p1}, it is straightforward to check the following
structure theorem. Every (theta,wheel)-free graph can be constructed as follows:
\begin{itemize}
\item Start with line graphs of triangle-free chordless graphs and P-graphs.
\item Repeatedly use consistent 2-join compositions from previously
  constructed graphs.
\item Gluing along a clique previously constructed graphs.
\end{itemize}

\end{document}